\PassOptionsToPackage{dvipsnames}{xcolor} 
\documentclass[11pt,a4paper]{article}
\usepackage{amssymb,amsmath,amsfonts,amscd,amsthm,amsxtra,latexsym,stmaryrd,scalerel}
\usepackage[utf8]{inputenc} 
\usepackage[english]{babel} 
\usepackage[T1]{fontenc}    
\usepackage[%
  final,%
  bookmarksopen=true,
  hypertexnames=false, 
]{hyperref}
\hypersetup{
  linkcolor = violet!85!black,
  citecolor = YellowOrange!85!black,
  urlcolor = Aquamarine!85!black,
  colorlinks=true
}
\usepackage{xcolor}
\usepackage{textcase}
\theoremstyle{plain}
\newtheorem{theorem}{Theorem}
\newtheorem{proposition}{Proposition}
\newtheorem{lemma}{Lemma}
\newtheorem{corollary}{Corollary}
\newtheorem{definition}{Definition}

\theoremstyle{remark}

\newtheorem{remark}{Remark}
\newtheorem{example}{Example}

\,
\newcommand{\parens}[1]{\ensuremath{\left(#1\right)}}
\newcommand{\set}[1]{\ensuremath{\left\{#1\right\}}}
\newcommand{\suchthat}{\,\middle|\,}

\newcommand{\I}[1]{\ensuremath{\mathbb{#1}}}
\newcommand{\Exterior}{\mathchoice{{\textstyle\bigwedge}}%
    {{\bigwedge}}%
    {{\textstyle\wedge}}%
    {{\scriptstyle\wedge}}}

\newcommand{\lie}[1]{\ensuremath{\mathfrak{\MakeTextLowercase{#1}}}}

\DeclareMathOperator{\Spin}{Spin}
\DeclareMathOperator{\spin}{\mathfrak{spin}}
\DeclareMathOperator{\U}{U}
\let\u\undefined \DeclareMathOperator{\u}{\mathfrak{u}}

\DeclareMathOperator{\lieG2}{\lie{g}_2}

\newcommand{\GL}{\mathrm{GL}}
\newcommand{\SO}{\mathrm{SO}}

\newcommand{\SU}{\mathrm{SU}}

\newcommand{\Sp}{\mathrm{Sp}}
\newcommand{\so}{\mathfrak{so}}
\newcommand{\su}{\mathfrak{su}}
\newcommand{\gl}{\mathfrak{gl}}
\renewcommand{\sp}{\mathfrak{sp}}

\newcommand{\Sym}{\mathrm{Sym}}

\newcommand{\cyclic}{\mathfrak{S}}

\newcommand{\R}{\mathrm{I\!R}}
\newcommand{\C}{\mathbb{C}}

\newcommand{\frakg}{\mathfrak{g}}

\newcommand{\dr}{\mathrm{dr}}

\,

\newcommand{\ad}{\mathrm{ad}}

\newcommand{\Ric}{\mathrm{Ric}}

\renewcommand{\d}{\mathrm{d}}

\newcommand{\End}{\mathrm{End}}
\newcommand{\SL}{\mathrm{SL}}
\newcommand{\trace}{\mathrm{tr}}

\newcommand{\scal}{\mathrm{s}}

\newcommand{\rmP}{\mathrm{P}}

\newcommand{\rmS}{\mathrm{S}}
\newcommand{\rmU}{\mathrm{U}}

\newcommand{\scrC}{\mathcal{C}}
\newcommand{\scrW}{\mathcal{W}}

\newcommand{\bbS}{\mathbb{S}}

\def\pd{\partial}

\title{A special class of symmetric Killing 2-tensors}
\author{Konstantin Heil and Tillmann Jentsch}
\begin{document}\sloppy

\maketitle
\begin{abstract}
We study symmetric Killing 2-tensors on Riemannian manifolds and show that several additional conditions can be realised only for Sasakian manifolds and Euclidean spheres. 
In particular we show that (three)-Sasakian manifolds can also be characterized by properties of the symmetric products of their characteristic 1-forms. 
Moreover, we recover a result of S.~Gallot on the characterization of spheres by means of functions satisfying a certain differential equation of order three.
\\

\noindent


\end{abstract}

\section{Introduction}
Symmetric tensors on a Riemannian manifold $(M^n,g)$ are by definition invariant under an arbitrary interchange of indices. On the other hand, (alternating) forms reverse their sign  whenever one interchanges two different indices. More generally,  the symmetry type of a tensor of valence $d$ reflects its behaviour under the action of the symmetric group $\rmS_d$ on its indices. Another famous example are the algebraic identities satisfied by the Riemannian curvature tensor like the first Bianchi identity. 

According to Schur-Weyl duality, symmetry types of tensors are in one-one correspondence with irreducible representations of the special linear group $\SL(n)$ (cf.~\cite[Theorem 6.3]{FH}). The Killing equation  in its general form is then stated as follows. Every function is by definition Killing, i.e. the Killing condition is empty for sections of the trivial vector bundle $M\times\R$. If $V$ is a non-trivial irreducible representation of the (special) linear group and  $\gamma$ is a section of the vector bundle $VM$ associated to the frame bundle of $M$, then $\nabla \gamma$ is a section of $TM^* \otimes VM$. Decomposing the latter into irreducible subbundles  according to the branching rules of Littlewood-Richardson (cf.~\cite[App.~A]{FH}), the Killing equation demands the vanishing of the Cartan component of $\nabla \gamma$.

Because the Killing equation is of finite type, the existence and uniqueness of the prolongation of the Killing equation is a priori ensured from an abstract point of view (cf.~\cite{BCEG}). Thus there exists a vector bundle $\mathrm{Prol}(VM) \to M$ equipped with a linear connection and a linear assignment $\Gamma(VM)\to \Gamma(\mathrm{Prol}(VM))$ such that Killing tensors correspond with parallel sections of $\mathrm{Prol}(VM)$.

\subsection{Killing forms}
For example, the Killing equation for $p$-forms is $\d\omega = (p+1) \nabla\omega$ for all $p\geq 0$, where $\d$ and $\nabla$ denote the exterior differential and the Levi-Civita connection.
In other words, $\omega$ is Killing if and only if $\nabla\omega$ is alternating, again. 
In~\cite{Se} U.~Semmelmann studied more generally conformal Killing forms on Riemannian manifolds and gave partial classifications. 
In particular, he found the complete classification of so called {\em special} Killing forms on compact Riemannian manifolds.
A Killing form $\omega$ is called special if it satisfies $\nabla_X \d \omega = - c\, X^\sharp\wedge \omega$ 
for some constant $c > 0$, where $X^\sharp$ denotes the dual of a vector field $X$ (cf.~\cite[Chapter~3]{Se}). 
To every pair $(\omega,\eta)$ of a $p$-form $\omega$ and a $p+1$-form $\eta$ on $M$ we associate the $p+1$-form $r^p \d r \wedge \omega + \frac{r^{p+1}}{p+1}  \eta$  
on the cone $\hat M := M\times \R_+$  (cf.~\cite[3.2.4]{Se}).
Then it is easy to see that $\omega$ is a special Killing form if and only if  the $p+1$-form $\hat\omega := \frac{1}{p+1} \d ( r^{p+1} \omega)$ associated with the pair $(\omega,\d \omega)$ is Levi-Civita parallel with respect to the cone metric. 
On the other hand, a parallel $p+1$-form on the cone yields a special Killing $p$-form on $M$ by inserting the
radial vector field at $r=1$. Thus there is a 1-1 correspondence between parallel $p+1$-forms on the cone and special Killing $p$-forms on $M$. Using a classical
theorem of S.~Gallot (cf.~\cite[Proposition~3.1]{Ga}) one obtains the classification of special Killing forms on compact Riemannian
manifolds via the holonomy principle (cf.~\cite[Theorem~3.2.6]{Se}). Besides the usually considered examples (euclidean spheres and Sasakian manifolds), this
includes also certain exceptional geometries in dimensions six and seven. Moreover, specifying $p = 0$, i.e. $\omega$ is a function $f$, we recover that a non-vanishing function $f$  satisfying the second order differential equation $\nabla^2 f = n\cdot g\cdot f$ (denoted by $(E_1)$ in~\cite{Ga}) exists only on the unit-sphere.

Every Killing $p$-form $\omega$ also satisfies the second order equation $\nabla_X\d\, \omega = \frac{p+1}{p}R^+(X)\omega$, the prolongation of the Killing equation. Here $R^+(X)$ denotes the natural action of degree one  of the endomorphism valued 1-form $R_{X,\bullet}$ on forms (cf.~\cite[Chapter~4]{Se}). In particular, a form $\omega$ is Killing if and only if
the prolongation $(\omega,\d \omega)$ is parallel with respect to a natural linear connection on the corresponding vector bundle, the Killing
connection. Moreover, on the unit sphere we have $R^+(X)\omega = - p\,X^*\wedge \omega$. In other words, a Killing form is special if and only if the prolonged Killing
equation has the same form as generally for Killing forms on a round sphere.

\subsection{A brief overview on our results for symmetric tensors}
Since the Killing equation for 1-forms (i.e. Killing vector fields) is well understood, we focus on symmetric 2-tensors:

\bigskip
\begin{definition}
A symmetric 2-tensor $\kappa$ is Killing if the completely symmetric
part of $\nabla\kappa$ vanishes, i.e.
\begin{equation}\label{eq:def_Killing}
\forall p\in M,x\in T_pM:\;\nabla_x\kappa(x,x)\;\;=\;\;0.
\end{equation}
\end{definition}

The crucial difference when compared with forms comes from the form of the prolongation and the lift to the cone associated with it. As laid out before, we can see the pair $(\omega,\d\omega)$ associated with a form $\omega$ not only as a set of variables of the prolongation, but also as a form on the  cone. 
Quite differently, for symmetric 2-tensors the variables of the prolongation are triples $(\alpha,\beta,\gamma)\in \bbS_{T_0}V^*\oplus \bbS_{T_1}V^*\oplus\bbS_{T_2}V^*$ of tensors of different symmetry types described by Young tableaus $T_i$ of shapes $(2,i)$ for $i =0,1,2$,
see~\eqref{eq:kappa_prolong_1}-\eqref{eq:kappa_prolong_2} and~\eqref{eq:def_C_kappa}-\eqref{eq:def_S_kappa} below. 
Such a triple can be  naturally seen as the components of a (symmetrized) algebraic curvature tensor on the cone via $S :=  r^2\, \alpha\owedge \d r\bullet \d r + r^3\, \beta\owedge \dr  + r^4\, \gamma$, where $\bullet$ is the usual symmetric product and $\owedge$ is a product which will be introduced in the next section. We will understand under which conditions $S$ is Levi-Civita parallel on the metric cone (see Theorem~\ref{th:main_1}) and relate these conditions to the geometry of $M$ (see Theorem~\ref{th:main_2}). Finally, given a function $f$ on $M$ which satisfies the partial differential equation $(E_2)$ of order three considered in~\cite{Ga} (see~\eqref{eq:Gallots_Gleichung} below), 
we associate with $f$  a Killing tensor which fits into Theorem~\ref{th:main_1}. Thus we recover the main result of~\cite{Ga} on the characterization of the unit sphere by the existence of such functions 
(see Corollary~\ref{co:Gallot}). 

Even though already for valence two these constructions are most efficently formulated  by means of
Young symmetrizers, explicit descriptions avoiding this formalism are possible,
see for example~\cite{HM} (for valence two) and~\cite{Th,W} (for arbitrary valence). For symmetric tensors of arbitrary valence, Y.~Houri and others~\cite{HTY} have given an explicit construction of the Killing prolongation via a recursive construction 
using the technique of Young symmetrizers. Although their arguments can not immediately be generalised to other symmetry types, 
nevertheless one would conjecture that similar explicit constructions are possible for all symmetry types of tensors.

\section{Weyls construction of irreducible representations of the general linear group}
\label{se:introduction_to_Young_symmetrizers}
We recall the notion of Young tableaus, the action of the associated symmetrizers and projectors on tensor spaces 
and their relation to  irreducible representations of the general linear group via the Schur functor. 
The notation is in accordance with~\cite{Sch}, for details see Fulton-Harris~\cite{F,FH}. 


A partition $\lambda_1\geq \cdots \geq \lambda_k >0$ of an integer $d$ can be depicted through a
Young frame, an arrangement of $d$ boxes aligned from the left in $k$-rows of length $\lambda_i$ counted from top to bottom. 
For example the frame corresponding to $(5,3,2)$ is
$$%
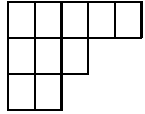
\begin{array}{|c|c|c|c|c|}
\cline{1-5}
 &  &  &  &  \\
\cline{1-5} 
 &  &  & \multicolumn{2}{c}{\;\;}\\
\cline{1-3}
 &  &  \multicolumn{3}{c}{\;\;}\\
\cline{1-2}
\end{array}
 $$
Filling the diagram with $d$ different numbers $\{i_1,\ldots,i_d\}$ we
obtain a {\em Young tableau} of {\em shape} $\lambda$ (cf.~\cite{F}). For example,
\begin{equation}\label{eq:normal_standard_Young_tableau}
T= \begin{array}{|c|c|c|c|c|}
\cline{1-5}
1 & 10 & 9 & 2 & 5\\
\cline{1-5}
8 & 7 & 4 & \multicolumn{1}{c}{\;\;} \\
\cline{1-3}
3 & 6 & \multicolumn{2}{c}{\;\;}\\
 \cline{1-2}
\end{array}
\end{equation}
 is a Young tableau of shape $(5,3,2)$.  For simplicity we assume 
in the following that $\{i_1,\ldots,i_d\} = \{1,\ldots,d\}$. When these
numbers are in order, left to right and top to bottom, the tableau is called
{\em normal.} Since there is only one normal diagram of a given shape $\lambda$, we will denote this by
 $T_\lambda$. 

Let $V^n$ be some vector space with dual space $V^*$. 
The tensor product $\bigotimes^d V^*$ can be seen as the space of multilinear forms of degree $d$ of $V$.
Here we have the natural right action of the symmetric group $\rmS_d$ given by
\begin{equation*}
\lambda \cdot \sigma(v_1, \cdots, v_d) := \lambda(v_{\sigma^{-1}(1)}, \cdots,v_{\sigma^{-1}(d)}).
\end{equation*}
Let $\rmS_c$ and $\rmS_r$ denote the
subgroup of $\rmS_d$ preserving columns and rows, respectively, of some fixed Young tableau $T$ of shape $\lambda$. The row symmetrizer 
and the column anti-symmetrizer  are
\begin{align}\label{eq:def_r_lambda}
&r_T \colon \bigotimes^dV^* \to \bigotimes^dV^*,\ \lambda \mapsto 
\sum_{\sigma\in \rmS_r} \lambda\cdot \sigma,\\
\label{eq:def_c_lambda}
&c_T \colon \bigotimes^dV^*\to \bigotimes^dV^*,\ \lambda \mapsto \sum_{\sigma\in \rmS_c} (-1)^{|\sigma|} \lambda\cdot\sigma.
\end{align}
The Young symmetrizer and its adjoint associated to a tableau $T$ are defined by 
\begin{align}
\label{eq:def_S_*_lambda}
&\rmS_T \;\;:=\;\; r_T\circ c_T,\\
\label{eq:def_S_lambda}
&\rmS^\star_T \;\;:=\;\;  c_T\circ r_T.
\end{align}
The endomorphisms $r_T$ and $c_T$ on
$\bigotimes^dV^*$ are called intertwining maps.

The images $\bbS_TV^* := \rmS_T(\bigotimes^dV^*)$ and $\bbS^\star_TV^* :=
\rmS^\star_T(\bigotimes^dV^*)$ both are irreducible representations of the
general linear group $\GL(n)$, dual to the one of highest weight $\lambda$. The intertwining maps yield by construction explicit invariant isomorphisms between $\bbS_TV^*$ and $\bbS^\star_TV^*$. 
The assignment $V \mapsto \bbS_TV^*$ (or $V \mapsto \bbS^\star_TV^*$) is called the Schur functor for covariant tensors associated with $T$.

In particular, by means of Schurs lemma there exists a constant $h$ such that $\rmP_T := \frac{1}{h_\lambda} \rmS_T$ and $\rmP^\star_T := \frac{1}{h_\lambda} \rmS^\star_T$ both are projectors, the Young projectors associated with $T$.
This constant is actually an integer which depends only on 
the underlying Young frame and hence we can write $h = h_\lambda$. It is given as follows:

A {\em hook} of length $d$ is a  Young frame with $d$ boxes but only one row and one column of length larger than
one. The box in the first row on the uttermost left is called its center.  For example, a hook of length $4$ is given by
\begin{equation}\label{eq:hook}
\begin{array}{|c|c|c|}
\cline{1-3}
 &  &  \\
\cline{1-3}
 & \multicolumn{2}{c}{\;\;}\\
 \cline{1-1}
\end{array}
\end{equation}

For every box of $T$ there is a unique maximal hook inscribed into the 
diagram centered at the given box. Its length is called the hook length of the
box. Then  $h_\lambda$ is the product of all hook numbers taken over all boxes of the frame. 

\paragraph{}
Following~\cite[Ch.~15.5]{FH}, there is another characterization of the representation space $\bbS^\star_TV^*$. Let $\mu_1\geq \cdots \geq \mu_\ell$ denote the conjugate
partition, i.e. column lengths. Then $\bbS^\star_TV^*$ is by construction a subspace of $\Lambda^{\mu_1}V^*\otimes\cdots \otimes \Lambda^{\mu_\ell}V^*$. Further, as a consequence of the branching rules due
to Littlewood-Richardson, it is immediately clear that $\bbS^\star_TV^*$ is a subspace of the kernel of the bilinear map
\begin{align}
&\ell_{ij}^\star \colon \Lambda^{\mu_i}V^*\times \Lambda^{\mu_j}V^*\to
\Lambda^{\mu_i+1}V^*\otimes \Lambda^{\mu_j-1}V^*,\quad (\omega_i,\omega_j)
\mapsto \ell_{ij}^\star(\omega_i,\omega_j):\\ 
& \ell_{ij}^\star(\omega_i,\omega_j)(v_1,\ldots,v_{\mu_i+\mu_j})\;\;:=\;\; \sum_{a=1}^{\mu_i+1}(-1)^{a+\mu_i+1}\omega_i(v_1,\cdots,\hat v_a,\ldots,v_{\mu_i+1})\omega_j(v_a,v_{\mu_i+2},\ldots,v_{\mu_i+\mu_j}),
\end{align}
which anti-symmetrizes all the indizes of $\omega_i$ with one further
index of $\omega_j$ for $i<j$. Moreover, it is also easy to see that $\bbS^\star_TV^*$ is in fact equal to the intersection of the
Kernels of all $\ell_{ij}^\star$ taken over all $1\leq i < j \leq \ell$.

For $\gamma\in \bigotimes_{a=1}^\ell \Lambda^{\mu_a}V^*$ let
$\gamma_{ij} \colon \Lambda^{\mu_i}V\times \Lambda^{\mu_j}V \to \bigotimes_{a\neq i,j} \Lambda^{\mu_a}V^*$ be the natural map. Then the previous is equivalent to
\begin{equation}\label{eq:exchange_rule_column}
\gamma_{ij}(v_1,\cdots,v_{\mu_i},v_{\mu_i+1},\cdots,v_{\mu_i+\mu_j}) = \sum_{a=1}^{\mu_i}\gamma(v_1,\cdots,v_{a-1},v_{\mu_i+1},v_{a+1},\ldots,v_{\mu_i},v_a,v_{\mu_i+2},\ldots,v_{\mu_i+\mu_j}),
\end{equation}
i.e. the elements of $\bbS^\star_TV^*$ respect the Plücker relations described in~\cite[15.53]{FH}. In~\cite[Ch.~8]{F} the above equation is seen 
as an ``exchange rule'' of length one between the  $i$-th and $j$-th column of $T$. Exchange rules of greater lengths (like pair symmetry for algebraic curvature tensors) then follow
automatically, cf. also~\cite[Exercise 15.54]{FH}.

\bigskip
\begin{example}
Consider the tableau $T := \scaleto{\begin{array}{|c|c|}
\cline{1-2}
1 & 3 \\
\cline{1-2}
2 & 4 \\
\cline{1-2}
\end{array}}{16pt}$. The tensor products $\Lambda^2V^*\otimes \Lambda^2V^*$ and  $\Lambda^3V^*\otimes V^*$ decompose according to the Littlewood-Richardson rules as follows
\begin{align*}
\Lambda^2V^*\otimes \Lambda^2V^* & = \Lambda^4V^*\oplus \bbS^\star_{\scaleto{\begin{array}{|c|c|}
\cline{1-2}
1 & 3\\
\cline{1-2}
2 &  \multicolumn{1}{c}{} \\
\cline{1-1}
4 &  \multicolumn{1}{c}{}\\
\cline{1-1}
\end{array}}{12pt}} V^*\oplus  \bbS^\star_TV^*,\\
\Lambda^3V^*\otimes V^* &= \Lambda^4V^*\oplus \bbS^\star_{\scaleto{\begin{array}{|c|c|}
\cline{1-2}
1 & 3\\
\cline{1-2}
2 &  \multicolumn{1}{c}{} \\
\cline{1-1}
4 &  \multicolumn{1}{c}{}\\
\cline{1-1}
\end{array}}{12pt}} V^*.
\end{align*}
The first two factors are already skew-symmetric in the first three indices and hence $\ell_{12}^\star$ yields isomorphisms between the first and second summands in each decomposition, respectively. Therefore, it 
necessarily vanishes on $\bbS^\star_TV^*$. From this we see that $\bbS^\star_{\scaleto{\begin{array}{|c|c|}
\cline{1-2}
1 & 3 \\
\cline{1-2}
2 & 4 \\
\cline{1-2}
\end{array}}{12pt}}V^*$ describes the subspace of $\Lambda^2V^*\otimes \Lambda^2V^*$ given by those tensors which satisfy the first Bianchi identity, i.e algebraic prototypes of the
curvature tensor $\tilde R(x_1,x_2,x_3,x_4)$ of some Riemannian manifold
$\tilde M$ at some point $p\in \tilde M$ (cf. also~\cite[p.12]{Sch}).  The
only exchange rule of order two leads to pair symmetry, which is a well
known algebraic consequence of the first Bianchi identity in $\Lambda^2V^*\otimes \Lambda^2V^*$.
\end{example}

Dividing through the action of the determinant (which acts as a scalar
according to Schurs lemma), we obtain an irreducible representation of the special linear
group $\SL(n)$ and, conversely, up to isomorphy  every irreducible representation of $\SL(n)$ is obtained in this way (cf.~\cite[Ch.~15]{FH}). 

The representation space $\bbS_TV^*$ has a description completely analogous to  $\bbS^\star_TV^*$ (although this fact seems to be less popular).
By definition, $\bbS_TV^*$ is a subspace of $\Sym^{\lambda_1}V^*\otimes\cdots \otimes \Sym^{\lambda_k}V^*$ (where as before $\lambda_i$ are the lengths of the rows of $T$).
Then it follows by similar arguments as above that $\bbS_TV^*$ is given by the
intersection of the kernels of the linear maps
\begin{align}
& \ell_{ij} \colon \Sym^{\lambda_i}V^*\otimes \Sym^{\lambda_j}V^*\to \Sym^{\lambda_i+1}V^*\otimes
\Sym^{\lambda_j-1}V^*,\quad (\alpha_i,\alpha_j)\mapsto \ell_{ij}(\alpha_i,\alpha_j):\\
&\ell_{ij}(\alpha_i,\alpha_j)(v_1,\ldots,v_{\lambda_i+\lambda_j})\;\;:=\;\;
\sum_{a=1}^{\lambda_i+1}\alpha_i(v_1,\cdots,\hat v_a,\ldots,v_{\lambda_i+1})\alpha_j(v_a,v_{\lambda_i+2},\ldots,v_{\lambda_i+\lambda_j})
\end{align}
with $i< j$. 
As before, exchange rules between rows of greater length follow automatically.

\bigskip
\begin{example}\label{ex:Young_symmetrizer}
\begin{enumerate}
\item For every covariant 2-tensor $\alpha$ \begin{equation}\label{eq:Young_(2,0)}
\rmS_{\scaleto{\begin{array}{|c|c|}
\cline{1-2}
1 & 2 \\
\cline{1-2}
\end{array}}{6pt}}\alpha(v_1,v_2) \;\;= \;\;
\begin{array}{c}
\alpha(v_1,v_2)+\alpha(v_2,v_1)
\end{array}
\end{equation} 
and $h_{(2)} = 2\cdot 1$.  Hence $\bbS_{\scaleto{\begin{array}{|c|c|}
\cline{1-2}
1 & 2 \\
\cline{1-2}
\end{array}}{6pt}}V^*$ is the space of covariant symmetric 2-tensors and the Young projector $\rmP_{\scaleto{\begin{array}{|c|c|}
\cline{1-2}
1 & 2 \\
\cline{1-2}
\end{array}}{6pt}}$ is the usual projector on symmetric 2-tensors.
\item
 We have \begin{equation}\label{eq:Young_(2,1)}
\rmS_{\scaleto{\begin{array}{|c|c|}
\cline{1-2}
2 & 3 \\
\cline{1-2}
1 & \multicolumn{1}{c}{} \\
 \cline{1-1}
\end{array}}{12pt}}\beta(v_1,v_2,v_3)\;\;= \;\; \left \lbrace
\begin{array}{c}
\;\  \beta(v_1,v_2,v_3) - \beta(v_2,v_1,v_3)\\
+ \beta(v_1,v_3,v_2) - \beta(v_3,v_1,v_2)
\end{array}\right \rbrace.
\end{equation} 
Further, $h_{(2,1)} = 3\cdot 1\cdot 1$. Hence the following is equivalent:
\begin{itemize}
\item
$
\beta \in \bbS_{\scaleto{\begin{array}{|c|c|}
\cline{1-2}
2 & 3 \\
\cline{1-2}
1 & \multicolumn{1}{c}{} \\
 \cline{1-1}
\end{array}}{12pt}} V^*$,
\item
$\rmS_{\scaleto{\begin{array}{|c|c|} 
\cline{1-2}
2 & 3 \\
\cline{1-2}
1 & \multicolumn{1}{c}{} \\
 \cline{1-1}
\end{array}}{12pt}}\beta(v_1,v_2,v_3) \;\;= \;\; 3\, \beta(v_{1},v_{2},v_{3})$,
\item 
$\beta\in V^* \otimes \Sym^2 V^*$ and
such that the completely symmetric part $\beta(v,v,v)$ vanishes. 
\end{itemize}
Such tensors are the algebraic prototype of the
covariant derivative $\nabla_{x_1}\kappa(x_2,x_3)$  of a symmetric Killing 2-tensor at a given point of some Riemannian manifold.
\item  We have $ \rmS_{\scaleto{\begin{array}{|c|c|}
\cline{1-2}
1 & 2 \\
\cline{1-2}
3 & 4 \\
\cline{1-2}
\end{array}}{12pt}}\gamma(v_1,v_2,v_3,v_4) =$ \begin{equation}\label{eq:Young_(2,2)}
 \left \lbrace \begin{array}{c}\;\ \gamma(v_1,v_2,v_3,v_4) - \gamma(v_3,v_2,v_1,v_4) - \gamma(v_1,v_4,v_3,v_2)  + \gamma(v_3,v_4,v_1,v_2)\\
 + \gamma(v_2,v_1,v_3,v_4)  - \gamma(v_3,v_1,v_2,v_4)  - \gamma(v_2,v_4,v_3,v_1)  + \gamma(v_3,v_4,v_2,v_1)\\
+ \gamma(v_1,v_2,v_4,v_3)  - \gamma(v_4,v_2,v_1,v_3)  - \gamma(v_1,v_3,v_4,v_2)  + \gamma(v_4,v_3,v_1,v_2)\\
 + \gamma(v_2,v_1,v_4,v_3)  - \gamma(v_4,v_1,v_2,v_3)  - \gamma(v_2,v_3,v_4,v_1)  + \gamma(v_4,v_3,v_2,v_1)\\
\end{array}
\right \rbrace.
\end{equation}
Further, $h_{(2,2)} = 3\cdot 2\cdot 2\cdot 1$. Therefore the following is equivalent:
\begin{itemize}
\item  $\gamma \in \bbS_{\scaleto{\begin{array}{|c|c|}
\cline{1-2}
1 & 2 \\
\cline{1-2}
3 & 4 \\
\cline{1-2}
\end{array}}{12pt}}V^*$,
\item   
$
S_{\scaleto{\begin{array}{|c|c|}
\cline{1-2}
1 & 2 \\
\cline{1-2}
3 & 4 \\
\cline{1-2}
\end{array}}{12pt}}\gamma(v_1,v_2,v_3,v_4) \;\;= \;\;  12\, \gamma(v_1,v_2,v_3,v_4)
$
\item $\gamma\in \Sym^2(V^*)\otimes\Sym^2(V^*)$ and $\gamma(u,u,u,v) = 0$ for all
  $u,v\in V$.
\end{itemize}
Then $\gamma$ is called a symmetrized algebraic curvature tensor (see~\cite[Example~4.10]{Sch}). Via the isomorphism $\bbS_{\scaleto{\begin{array}{|c|c|}
\cline{1-2}
1 & 2 \\
\cline{1-2}
3 & 4 \\
\cline{1-2}
\end{array}}{12pt}}V^* \cong \bbS^\star _{\scaleto{\begin{array}{|c|c|}
\cline{1-2}
1 & 3 \\
\cline{1-2}
2 & 4 \\
\cline{1-2}
\end{array}}{12pt}}V^*$ this is the algebraic prototype of the symmetrization 
\begin{equation}\label{eq:symmetrized_curvature_tensor}
\forall x_1,x_2,x_3,x_4\in T_p\tilde M:\; \gamma(x_1,x_2,x_3,x_4) := \cyclic_{12} \tilde R(x_1,x_3,x_2,x_4)
\end{equation} of the  curvature tensor $\tilde R$ of some Riemannian manifold. 
\end{enumerate}
\end{example}

\subsection{A generalisation of the wedge  product}
\label{se:product}
Let $U$, $V$ and $W$ be irreducible representations of the general linear
group $\GL(n)$. If $W$ occurs with multiplicity one in the decomposition of $U\otimes V$, then there exists an invariant bilinear map
$\owedge \colon U\times V \to W$ which is unique up to a factor by Schurs Lemma. In the cases relevant for our article,
the underlying Young frame of all three involved partitions has at most two
rows. Let $\lambda = (\lambda_1,\lambda_2)$, $\mu=(\mu_1,\mu_2)$  and $\nu =
(\nu_1,\nu_2)$ be partitions of length at most two. Further, suppose that
$\lambda_1 \leq \nu_1\leq \lambda_1 + \mu_1$, $\lambda_2\leq \nu_2$ and that the length of these partitions satisfies $|\nu|= |\lambda| + |\mu|$. 
Let $T_1$, $T_2$ and $T_3$ be tableaus of shape $\lambda$, $\mu$ and $\nu$, respectively.
Then $\bbS_{T_3}V^*$ occurs with multiplicity one in the decomposition of the tensor product 
$\bbS_{T_1}\,V^* \otimes \bbS_{T_2}\,V^*$ according to the
Littlewood-Richardson rules.  Hence, because of Schurs lemma there is up to
a factor a unique invariant bilinear map 
\begin{equation}
\owedge \colon \bbS_{T_1}\,V^* \times \bbS_{T_2}\,V^*\to \bbS_{T_3}V^*.
\end{equation}
In order to reduce the occurence of annoying factors in 
our formulas, we make the following specific choice of $\owedge$: suppose for simplicity that all three
tableaus are normal. 
Then $\lambda \owedge \mu(v_1,\cdots, v_{\nu_1 + \nu_2})$ is obtained by applying the Young symmetrizer
$\rmS_{T_3}$ to the product $\lambda(v_1,\cdots,v_{\lambda_1},v_{\nu_1+1},\cdots,v_{\nu_1 + \lambda_2})
\mu(v_{\lambda_1+1},\cdots, v_{\nu_1},v_{\nu_1 + \lambda_2 + 1},\cdots,v_{\nu_1 + \nu_2})$ 
and dividing by the product $h_\lambda h_\mu$ of all hook numbers of the two
frames of shape $\lambda$ and $\mu$, this way generalising the standard
definition of the wedge product $\omega\wedge \eta := \frac{n!}{p!q!} \mathrm{Alt}(\omega\otimes
\eta)$ of $p$- and $q$-forms $\omega$ and $\eta$ from
differential geometry. Since our plain notation $\owedge$ is still ambiguous, we describe explicitly the products used in the sequel:

\bigskip
\begin{definition}\label{de:product}
\begin{enumerate}
\item We keep to the standard notation
\begin{equation}
\lambda_1\bullet \lambda_2(v_1,v_2) \;\;:=\;\; \lambda_1(v_1)\lambda_2(v_2) + \lambda_1(v_2)\lambda_2(v_1)
\end{equation}
for the symmetrized tensor product of 1-forms $\lambda_1$ and $\lambda_2$ (cf.~\cite{HMS}).
\item Given a 1-form $\lambda$ and a symmetric 2-tensor $\alpha$, we set 
\begin{equation}\label{eq:product_of_a_one_form_and_a_symmetric_2-tensor}
\lambda\owedge \alpha(v_1,v_2,v_3) \;\;:=\;\; -\alpha\owedge \lambda (v_1,v_2,v_3)\;\;:=\;\;\frac{1}{2}\, \rmS_{\scaleto{\begin{array}{|c|c|}
\cline{1-2}
2 & 3 \\
\cline{1-2}
1 & \multicolumn{1}{c}{} \\
 \cline{1-1}
\end{array}}{12pt}}\, \lambda(v_2)\alpha(v_1,v_3).
\end{equation}
\item We define the Cartan product
\begin{equation}\label{eq:def_curvature_tensor_associated_with_two_2-forms}
\omega_1\owedge \omega_2(v_1,v_2,v_3,v_4)\;\;:=\;\; \frac{1}{4}\, \rmS_{\scaleto{\begin{array}{|c|c|}
\cline{1-2}
1 & 2 \\
\cline{1-2}
3 & 4\\
\cline{1-2} 
\end{array}}{12pt}}\, \omega_1(v_1,v_3)\omega_2(v_2,v_4)
\end{equation}
of alternating 2-forms $\omega_1$ and $\omega_2$.
\item Given a 1-form $\lambda$ and $\beta\in \bbS_{\scaleto{\begin{array}{|c|c|}
\cline{1-2}
2 & 3 \\
\cline{1-2}
1 & \multicolumn{1}{c}{} \\
 \cline{1-1}
\end{array}}{12pt}}V^*
$, we define 
\begin{equation}\label{eq:noch_ein_produkt}
\lambda  \owedge \beta(v_1,v_2,v_3,v_4)\;\;:=\;\;- \beta\owedge\lambda(v_1,v_2,v_3,v_4)\;\;:=\;\;-\frac{1}{3}\,\rmS_{\scaleto{\begin{array}{|c|c|}
\cline{1-2}
1 & 2 \\
\cline{1-2}
3 & 4\\
\cline{1-2} 
\end{array}}{12pt}}\, \lambda(v_1)\beta(v_2,v_3,v_4).
\end{equation}
\item 
We define a variant of the Kulkarni-Nomizu product
\begin{equation}\label{eq:def_Kulkarni_Nomizu_product}
\alpha_1\owedge \alpha_2(v_1,v_2,v_3,v_4)\;\;:=\;\;\frac{1}{4}\,\rmS_{\scaleto{\begin{array}{|c|c|}
\cline{1-2}
1 & 2 \\
\cline{1-2}
3 & 4\\
\cline{1-2} 
\end{array}}{12pt}}\, \alpha_1(v_1,v_2)\alpha_2(v_3,v_4)
\end{equation}
for symmetric 2-tensors $\alpha_1$ and $\alpha_2$.
\end{enumerate}
\end{definition}


\section{The main results}
Although we are mainly interested in the Riemannian case, we consider more
generally a pseudo Riemannian manifold $(M,g)$ of dimension $n$ with tangent
bundle $TM$, Levi-Civita connection $\nabla$ and curvature tensor
$R$. Further, we will use the standard notation $\langle x,y\rangle := g(x,y)$ for $x,y\in T_pM$.

The first and second (standard) variable of the prolongation of the Killing equation are defined by
\begin{align}\label{eq:kappa_prolong_1}
\kappa^1(x_1,x_2,x_3) \;:=\; & \rmP_{\scaleto{
\begin{array}{|c|c|}
\cline{1-2} 
2 & 3 \\
\cline{1-2}
1 & \multicolumn{1}{c}{\;\;\;}\\ 
\cline{1-1}
\end{array}}{12pt}} \nabla_{x_1}\kappa(x_2,x_3),\\
\label{eq:kappa_prolong_2}
\kappa^2(x_1,x_2,x_3,x_4) \;:=\; & \rmP_{\scaleto{\begin{array}{|c|c|}\cline{1-2}
1 & 2 \\
\cline{1-2}
3  & 4\\
\cline{1-2} 
\end{array}}{12pt}}\nabla^2_{x_1,x_2}\kappa(x_3,x_4).
\end{align}
In Section~\ref{se:prolongation} we will show that the triple $(\kappa,\kappa^1,\kappa^2)$ is in fact closed under componentwise covariant derivative, 
which yields the prolongation of the Killing equation for symmetric 2-tensors.

\bigskip
\begin{definition}
The (Riemannian) cone $\hat M$ is the trivial fibre bundle  $\tau \colon M\times \R_+ \to M$ whose total space is equipped with the metric tensor
\begin{equation}\label{eq.cone}
  \hat g := r^2\, g  +  \d r^2.
\end{equation}
\end{definition}
Then we see $M$ as a (pseudo) Riemannian submanifold of $\hat M$ via $\iota \colon M\to
\hat M,p\mapsto (p,1)$. Every symmetrized algebraic curvature tensor $S$ on $\hat M$ 
defines a symmetric 2-tensor $\kappa := \frac{1}{2} \pd_r\lrcorner\pd_r \lrcorner  S$ on $M$, i.e
\begin{equation}\label{eq:kanonische_konstruktion_von_killingtensoren}
\kappa(x,y) := \frac{1}{2} S(\pd_r|_p,\pd_r|_p,\d_p \iota\, x,\d_p \iota\, y)
\end{equation}
for all $p\in M$ and $x,y\in T_pM$. For a construction in the inverse
direction, let a symmetric 2-tensor $\kappa$ on $M$ be given:

\bigskip
\begin{definition}
The symmetrized algebraic curvature tensors defined by
\begin{align}
\label{eq:def_C_kappa}
C^\kappa &\;\;:=\;\; \kappa^2 +  \kappa \owedge g,\\
\label{eq:def_S_kappa}
S^\kappa&\;\;:=\;\; r^2\, \kappa\owedge \d r\bullet \d r + r^3\,
\kappa^1\owedge \dr  + r^4\, C^\kappa
\end{align}
will be called the associated symmetrized algebraic curvature tensors on $M$ and $\hat M$, respectively. 
\end{definition} 

We remark without proof that $S^\kappa$ has the alternative description $\rmP_{\scaleto{\begin{array}{|c|c|}
\cline{1-2}
1 & 2 \\
\cline{1-2}
3 & 4\\
\cline{1-2} 
\end{array}}{12pt}}\hat \nabla^2_{v_1,v_2}r^4\tau^*\kappa(v_3,v_4)$
.  


\bigskip
\begin{theorem}\label{th:main_1}
Let $M$ be a pseudo Riemannian manifold of dimension at least two with Levi-Civita
connection $\nabla$ and curvature tensor $R$. The definitions given in~\eqref{eq:kanonische_konstruktion_von_killingtensoren},~\eqref{eq:def_C_kappa} and~\eqref{eq:def_S_kappa} yield a 1-1 correspondence between 
\begin{enumerate}
\item symmetric Killing 2-tensors $\kappa$ on $M$ which satisfy the following nullity conditions
\begin{align}\label{eq:curv_cond_1}
R_{x,y}\kappa\;\;=\;\;& - x\wedge y \cdot \kappa\;,\\
\label{eq:curv_cond_2}
R_{x,y} \nabla \kappa\;\;=\;\;& - x\wedge y\cdot \nabla \kappa
\end{align}
for all $p\in M$ and $x,y\in T_pM$;
\item pairs $(\kappa,C)$ of symmetric 2-tensors $\kappa$ and symmetrized algebraic curvature tensors $C$ on $M$ satisfying
\begin{align}\label{eq:alg_curv_tensor_1}
C(x,\,\cdot\,,\,\cdot\,,\,\cdot\,)\;\;=\;\;& \nabla_x\nabla\kappa + 2\,\kappa\owedge x^\sharp,\\
\label{eq:alg_curv_tensor_2}
\nabla_x C \;\;=\;\;& - \nabla\kappa\owedge x^\sharp
\end{align}
for all $p\in M$ and $x \in T_pM$;
\item parallel symmetrized algebraic curvature tensors on the Riemannian cone $\hat M$.
\end{enumerate}
\end{theorem}

We would like to remark that there are always trivial solutions to~(a),~(b) and~(c) of the previous theorem coming from the metric tensor. 
For example a constant multiple $\kappa := c\, g$ of the metric tensor is in accordance with~\eqref{eq:curv_cond_1},~\eqref{eq:curv_cond_2}. Further, it is easy to see that the requirement that the dimension of $M$ 
has to be strictly larger than one is necessary, see Remark~\ref{re:main_3}.

\bigskip
\begin{example}\label{ex:constant_curvature}
If $M$ is simply connected and of constant (non-vanishing) sectional curvature, then we can realise $M$
as a standard model in the pseudo Euclidean space $(V^{n+1},\langle\ ,\ \rangle)$. Further,
possibly after reversing the sign of the metric of $M$ and scaling by a positive
constant,  we can assume that $M$ is the generalised unit sphere $\{v\in V|\langle v,v\rangle = 1\}$. Then the cone
over $M$ is $V\setminus\{0\}$ and the nullity
conditions~\eqref{eq:curv_cond_1} and~\eqref{eq:curv_cond_2} become
tautological. Thus algebraic curvature tensors on $V$ correspond with elements of $\bbS_{\scaleto{\begin{array}{|c|c|}
\cline{1-2}
1 & 2 \\
\cline{1-2}
3 & 4 \\
\cline{1-2}
\end{array}}{12pt}}V^*$ via~\eqref{eq:symmetrized_curvature_tensor}. We obtain
from $(a)\Leftrightarrow (c)$ that symmetric Killing 2-tensors on $M$  are in
1-1 correspondence with algebraic curvature tensors on $V$ (the flat case is handeled similarly).
Thus we have recovered the main result  of~\cite{McMS} in the
special case of 2-tensors.
\end{example}

\bigskip
\begin{remark}\label{re:main_2}
If a pair $(\kappa,C)$ satisfies Equations~\eqref{eq:alg_curv_tensor_1}-\eqref{eq:alg_curv_tensor_2}, 
then $C = C^\kappa$ and, vice versa, in dimension larger than one also $\kappa$ is determined by $C$, see~\eqref{eq:def_kappaC}-\eqref{eq:def_modified_Ricci_trace_und_tilde_scalC}.
\end{remark}

Further, every Riemannian manifold $(M,g)$ can also be seen as a pseudo Riemannian
manifold with negative definite metric tensor and vice versa
by considering $(M,-g)$. Hence the following remark addresses in fact compact
Riemannian manifolds: 

\bigskip
\begin{remark}\label{re:main_1}
On a compact  manifold with negative definite metric tensor there are
no solutions to the previous theorem besides the trivial ones.
\end{remark}

Here the compactness is essential, as becomes clear from considering the
hyperbolic space.

The proofs of Theorem~\ref{th:main_1} and the following two remarks are implicitly contained in 
Propositions~\ref{p:prolong},~\ref{p:main} and~\ref{p:get_rid} below.

\subsection{Characterization of Sasakian manifolds and spheres}
Next, we would like to discuss the relevance of Theorem~\ref{th:main_1} in the Riemannian case. When does a Riemannian
manifold admit tensors fitting (a),~(b) or~(c) other than the trivial solutions?

Recall that there are several ways to define a Sasakian manifold
$M$. One possibility is via the existence of a Kähler 
structure $I$ on $\hat M$ (cf.~\cite[Proposition~1.1.2]{BG}). The
corresponding parallel 2-from $\omega(u,v) := \langle I\, u,v\rangle$ is
called the Kähler form. Then the 1-form $\eta := \iota^*(\pd_r\lrcorner
\omega)$ on $M$ (defined by $\eta(x) := \omega(\pd_r|_p,\d_p\iota x)$ for all $p\in M$ and $x\in T_pM$) is called the characteristic form of the Sasakian manifold $M$. 

More specifically, one considers 3-Sasakian manifolds. Here the cone carries a Hyperkähler structure $\{I_1,I_2,I_3\}$. 
The corresponding Kähler forms $\{\omega_1,\omega_2,\omega_3\}$ induce 1-forms 
$\eta_i := \iota^*(\pd_r\lrcorner \omega_i)$ on $M$  as before called the characteristic forms 
of the 3-Sasakian manifold $M$ (cf.~\cite[Proposition~1.2.2]{BG}).

Examples of Sasakian and 3-Sasakian manifolds are round spheres in odd dimension and dimension 3 {\rm mod} 4, respectively.

In Section~\ref{se:sasaki} we will use S.~Gallots before mentioned theorem in
order to classify complete Riemannian manifolds whose 
cone carries a parallel algebraic curvature tensor by means of the holonomy principle. We obtain our second main result:

\bigskip
\begin{theorem}\label{th:main_2}
The classification of complete and simply connected Riemannian manifolds $(M,g)$ with respect to Theorem~\ref{th:main_1} is as follows.
\begin{itemize}
\item On a round sphere the space of symmetric Killing
  2-tensors is isomorphic to the linear space of algebraic curvature tensors
  in dimension $n+1$, see Example~\ref{ex:constant_curvature}.
\item  Suppose that $(M,g)$ is a 3-Sasakian manifold which is not of constant sectional curvature. 
Let $\{\eta_1,\eta_2,\eta_3\}$ denote its characteristic forms related to Kähler forms $\{\omega_1,\omega_2,\omega_3\}$ 
on the cone. Here the symmetrized tensor products $\eta_i\bullet\eta_j$ are symmetric Killing 2-tensors which correspond to 
the parallel symmetrized algebraic curvature tensors $\omega_i\owedge\omega_j$ on the cone. 
The pair $(\kappa,C^\kappa)$ with $\kappa := \eta_i\bullet\eta_j$ and $C^\kappa := \frac{1}{4} \d\eta_i\owedge\d\eta_j$ is a solution to~\eqref{eq:alg_curv_tensor_1}-\eqref{eq:alg_curv_tensor_2}. Together with the metric tensor we obtain a 7-dimensional space of symmetric 2-tensors matching for example~\eqref{eq:curv_cond_1},\eqref{eq:curv_cond_2} of Theorem~\ref{th:main_1}~a).
\item Suppose that $(M,g)$ is a Sasakian manifold  which is neither 3-Sasakian
  nor of constant sectional curvature one. Let $\eta$ denote its
  characteristic 1-form. Similar as in the 3-Sasakian case, the set of Killing
  tensors fitting into Theorem~\ref{th:main_1} is the 2-dimensional space spanned by the square $\eta \otimes \eta $ of the characteristic form and the metric tensor. 
\item Otherwise, there are only the obvious solutions to (a),~(b) or~(c)
  coming from the metric tensor.
\end{itemize}
\end{theorem}

We observe that all Killing tensors mentioned in the above theorem are  linear combinations of the metric tensor and symmetric products of Killing vector fields. 
Further, we recall that spheres and Sasakian manifolds are irreducible.

\bigskip
\begin{corollary}\label{co:main_1}
An even (odd) dimensional simply connected complete Riemannian manifold $M$ of
dimension at least two admits a non-trivial Killing tensor matching~\eqref{eq:curv_cond_1},\eqref{eq:curv_cond_2} of Theorem~\ref{th:main_1}~a) if and only if $M$ is the unit sphere (a Sasakian manifold).
\end{corollary}

Due to the fact that a Sasakian vector field has constant length, 
the square of the characteristic form of a Sasakian manifold has
constant trace. The same is true for the symmetric products of the characteristic forms of a
3-Sasakian manifold. Conversely, the Killing tensor $\kappa$ associated with a
parallel curvature tensor $S$ of $\R^{n+1}$ on an $n$-dimensional round sphere has $\d\, \trace
\kappa \neq 0$ unless the algebraic Ricci-tensor of $S$ is a multiple of
the Euclidean scalar product (i.e. unless $S$ has the algebraic
properties of the curvature tensor of an Einstein manifold). Therefore:

\bigskip
\begin{corollary}\label{co:main_2}
A simply connected complete Riemannian manifold $M$ carries a symmetric 2-tensor $\kappa$ matching~\eqref{eq:curv_cond_1},\eqref{eq:curv_cond_2} of Theorem~\ref{th:main_1}~a)
and such that $\d\, \trace \, \kappa \neq 0$ if and only if $M$ is a
round sphere.
\end{corollary}

Recall that one of the original goals of the paper~\cite{Ga} was to investigate the following differential
equation $(E_2)$ of order three on functions on Riemannian manifolds,
\begin{equation}\label{eq:Gallots_Gleichung}
\left . \begin{array}{c} \nabla^3_{y_1,y_2,y_3}f  +
    2\,\d f(y_1) \langle y_2,y_3\rangle  +
    \cyclic_{23}\langle y_1,y_2 \rangle\,
    \d f(y_3) \end{array} \right .\;\;=\;\;0.
\end{equation}

Let $\kappa$ be the symmetric 2-tensor 
\begin{equation}\label{eq:Gallots_kappa}
\kappa^f(x,y) := f\,\langle x ,y \rangle + \frac{1}{4} \nabla^2_{x,y} f.
\end{equation}
We will show in Section~\ref{se:Gallot} that the pair $(\kappa,C^\kappa)$ satisfies the the conditions of Theorem~\ref{th:main_1}~(b). Hence we can see the following result in the
context of symmetric Killing 2-tensors:

\bigskip
\begin{corollary}\label{co:Gallot}{\cite[Corollary~3.3]{Ga}}
Suppose that $M$ is a complete Riemannian manifold. 
If there exists a non-constant solution to~\eqref{eq:Gallots_Gleichung}, 
then the universal covering of $M$ is the Euclidean sphere of radius one.
\end{corollary}

We conjecture that the differential equation $(E_p)$ considered in Chapter~4 of the same paper~\cite{Ga} is related to symmetric Killing tensors of valence $p$ for all $p\geq 1$ in a similar way.
\section{The associated symmetrized algebraic curvature tensor on the cone}
\label{se:prolongation_on_the_cone}
Let $M$ be a (pseudo) Riemannian manifold and $\hat M$ its cone. The canonical projection $\tau:\hat M\to M$
is obtained by forgetting the second component of a point $(p,t)\in \hat M$. 
Further, the natural inclusion 
\[
\iota \colon M \to \hat M,p\mapsto (p,1)
\] 
exhibits $M$ as a Riemannian submanifold of $\hat M$. 
The tangent bundle of $\hat M$ splits as $\tau^*T \hat M = T M \oplus \R$. 
Vectors tangent to $M$ are called horizontal vectors. 
In the following $x,y,\ldots$ denote horizontal vectors and $u,v,\ldots$
arbitrary tangent vectors of $\hat M$.  The Levi-Civita connection is given by (cf.~\cite[Ch.~1]{Ga})
\begin{align}\label{eq:Levi-Civita_connection_of_the_cone_1}
& \hat\nabla_{\partial_r} \partial_r\;\;=\;\;0\;,\\
\label{eq:Levi-Civita_connection_of_the_cone_3}
& \hat\nabla_x y\;\;=\;\;\nabla_x y - r\, \langle x,y\rangle \partial_r,\\
\label{eq:Levi-Civita_connection_of_the_cone_2}
&\hat\nabla_{\partial_r} x\;\;=\;\;\hat\nabla_x {\partial_r}\;\;=\;\;\frac{1}{r} x.
\end{align}
The last two equations are precisely the equations of Gauß and Weingarten 
describing the Levi-Civita connection on the euclidean sphere of radius $r$ where $\pd_r$ becomes the outward pointing normal vector field.
Equation~\eqref{eq:Levi-Civita_connection_of_the_cone_1} implies that the vector field $\frac{x}{r}$ is parallel along the curve $r\mapsto (p,r)$  for each $x\in T_pM$ , i.e.
\begin{equation}\label{eq:parallele_vektorfelder_laengs_radialer_kurven}
\hat\nabla_{\pd_r}\frac{x}{r} = 0.
\end{equation}

Further, it follows that
\begin{align}\label{eq:hat_nabla_diff_r}
& \hat\nabla_{\partial_r}\d r\;\;=\;\;0\;,\\
\label{eq:hat_nabla_x_dr}
& \hat\nabla_x\d r\;\;=\;\;r\, x^\sharp
\end{align}
(where $x^\sharp$ denotes the dual 1-form $\langle x,\,\cdot\,\rangle $\,. Let $\gamma$ be  an arbitrary covariant $k$-tensor of $M$. 

\begin{itemize}
\item Given $x\in T_pM$ we define a covariant tensor $x\lrcorner \gamma$ of valence $k-1$ by setting
\begin{equation}\label{eq:x_einsetzen}
x\lrcorner \gamma (y_1,\cdots,y_{k-1}) := \gamma (x,y_1,\cdots,y_{k-1}).
\end{equation}
\item Given some endomorphism $A$ of $T_pM$ we define $A\cdot \gamma$ by 
\begin{equation}\label{eq:derivation}
A\cdot \gamma (x_1,\cdots,x_k) := - \sum_{i=1}^k \gamma(x_1,\cdots,A x_i,\cdots, x_k).
\end{equation}
\end{itemize}

If $\gamma$ is a covariant $p$-tensor on $M$, the pullback $\tau^*\gamma$ will 
also be denoted by $\gamma$. As a consequence of the previous we have:

\bigskip
\begin{lemma}
Let $\gamma$ and $\kappa$ be an arbitrary and a symmetric covariant $k$-tensor, respectively, on $M$.
\begin{align}\label{eq:hat_nabla_kappa_r}
& \hat\nabla_{\partial_r}\gamma\;\;=\;\;- \frac{k}{r} \gamma\;,\\
\label{eq:hat_nabla_gamma_x}
& \hat\nabla_{x}\gamma\;\;=\;\;\nabla_x \gamma - \frac{1}{r} \d r\otimes x \cdot \gamma,\\
\label{eq:hat_nabla_kappa_x}
& \hat\nabla_{x}\kappa\;\;=\;\;\nabla_x \kappa - \frac{1}{r}\d r \bullet x \lrcorner \kappa.
\end{align}
\end{lemma}

\bigskip
\begin{remark}\label{re:reduction}
Every basis of $T_pM$ can be extended to a basis of $T_{(p,r)}\hat M$ using the radial vector $\pd_r|_{(p,r)}$. 
Thus, the principal $\GL_{n+1}(\R)$ fiber bundle of frames of the vector bundle $T\hat M$ 
admits a natural reduction to the principal $\GL_n(\R)$ fiber bundle of frames of the pullback vector bundle $\tau^*TM$. In particular, $\nabla$ induces a connection on 
every tensor bundle of $T\hat M$. Therefore the latter decomposes into a direct sum of $\nabla$-parallel subbundles
according to the branching rules for the Lie algebra pair $\gl_n(\R)\subset \gl_{n+1}(\R)$.
\end{remark}

For the vector bundle $\bbS_{\scaleto{\begin{array}{|c|c|}
\cline{1-2}
1 & 2 \\
\cline{1-2}
3 & 4 \\
\cline{1-2}
\end{array}}{12pt}}\, T^*\hat M$ of symmetrized algebraic curvature tensors of the cone these
branching rules are given as follows: 

\bigskip
\begin{lemma}\label{le:branching_rules}
The vector bundle $\bbS_{\scaleto{\begin{array}{|c|c|}
\cline{1-2}
1 & 2 \\
\cline{1-2}
3 & 4 \\
\cline{1-2}
\end{array}}{12pt}}\, T^*\hat M$ admits a decomposition into three direct summands
\begin{equation}\label{eq:glndec}
\tau^*\bbS_{\scaleto{\begin{array}{|c|c|}
\cline{1-2}
1 & 2 \\
\cline{1-2}
\end{array}}{6pt}}\,T^*M\oplus \tau^*\bbS_{\scaleto{\begin{array}{|c|c|}
\cline{1-2}
2 & 3 \\
\cline{1-2}
1 & \multicolumn{1}{c}{\;\;\;} \\
\cline{1-1}
\end{array}}{12pt}}\,T^*M \oplus \tau^*\bbS_{\scaleto{\begin{array}{|c|c|}
\cline{1-2}
1 & 2 \\
\cline{1-2}
3 & 4 \\
\cline{1-2}
\end{array}}{12pt}}\, T^*M
\end{equation}
which is parallel with respect to the Levi-Civita connection of $M$. Suitable projection maps are given as follows:
\begin{eqnarray}
\label{eq:projection_1}
\bbS_{\scaleto{\begin{array}{|c|c|}
\cline{1-2}
1 & 2 \\
\cline{1-2}
3 & 4 \\
\cline{1-2}
\end{array}}{12pt}}\, T^*\hat M\to \tau^*\bbS_{\scaleto{\begin{array}{|c|c|}
\cline{1-2}
1 & 2 \\
\cline{1-2}
\end{array}}{6pt}}\,T^*M, & & S \mapsto \alpha^S \;\;:= \;\; \frac{1}{2}\pd_r\lrcorner \pd_r \lrcorner S\;,\\
\bbS_{\scaleto{\begin{array}{|c|c|}
\cline{1-2}
1 & 2 \\
\cline{1-2}
3 & 4 \\
\cline{1-2}
\end{array}}{12pt}}\, T^*\hat M\to \tau^*\bbS_{\scaleto{\begin{array}{|c|c|}
\cline{1-2}
2 & 3 \\
\cline{1-2}
1  & \multicolumn{1}{c}{\;\;\;}\\
\cline{1-1} 
\end{array}}{12pt}}T^*M \;,& & S \mapsto \beta^S  \;\;:= \;\; \iota^* \pd_r \lrcorner S\;,\\
\bbS_{\scaleto{\begin{array}{|c|c|}
\cline{1-2}
1 & 2 \\
\cline{1-2}
3 & 4 \\
\cline{1-2}
\end{array}}{12pt}}\, T^*\hat M \to \tau^*\bbS_{\scaleto{\begin{array}{|c|c|}
\cline{1-2}
1 & 2 \\
\cline{1-2}
3  & 4\\
\cline{1-2} 
\end{array}}{12pt}}T^*M \;,& & S\mapsto \gamma^S  \;\;:= \;\; \iota^*S.
 \end{eqnarray}
Then every triple $(\alpha,\beta,\gamma)$ defines the symmetrized algebraic curvature tensor 
\begin{equation}\label{eq:S_explizit}
S :=   \alpha \owedge\d r \bullet \d r + \beta \owedge\d r + \gamma.
\end{equation}
\end{lemma}
\begin{proof}
It is straightforward to verify~\eqref{eq:S_explizit}. Hence the map which projects from $\bbS_{\scaleto{\begin{array}{|c|c|}
\cline{1-2}
1 & 2 \\
\cline{1-2}
3 & 4 \\
\cline{1-2}
\end{array}}{12pt}}\, T^*\hat M $ to $\tau^*\bbS_{\scaleto{\begin{array}{|c|c|}
\cline{1-2}
1 & 2 \\
\cline{1-2}
\end{array}}{6pt}}\,T^*M \oplus\bbS_{\scaleto{\begin{array}{|c|c|}
\cline{1-2}
2 & 3 \\
\cline{1-2}
1 & \multicolumn{1}{c}{\;\;\;} \\
\cline{1-1}
\end{array}}{12pt}}\,T^*M \oplus \bbS_{\scaleto{\begin{array}{|c|c|}
\cline{1-2}
1 & 2 \\
\cline{1-2}
3 & 4 \\
\cline{1-2}
\end{array}}{12pt}}\,T^ *M $ along $\tau$ is surjective. In order to demonstrate its injectivity, suppose $\alpha^S = 0$, $\beta^S = 0$ and $\gamma^S = 0$. 
Then $S(u_1,u_2,u_3,u_4)$ vanishes whenever at least two of the vectors $u_i$ are horizontal by means of the symmetries of a curvature tensor.
On the other hand, if at most one vector is horizontal, then at least three
are proportional to $\partial_r$ and thus $S(u_1,u_2,u_3,u_4) = 0$, since $S$
satisfies the corresponding exchange rule of order one between rows. This proves the  decomposition~\eqref{eq:glndec}.
\end{proof}

Therefore symmetrized algebraic curvature tensors of $T_{(p,t)}\hat M$ should be seen as triples 
\[
(\alpha,\beta,\gamma) \in \bbS_{\scaleto{\begin{array}{|c|c|}
\cline{1-2}
1 & 2 \\
\cline{1-2}
\end{array}}{6pt}}\,T_p^*M  \oplus \bbS_{\scaleto{\begin{array}{|c|c|}
\cline{1-2}
2 & 3 \\
\cline{1-2}
1 & \multicolumn{1}{c}{\;\;\;} \\
\cline{1-1}
\end{array}}{12pt}}\, T_p^*M \oplus  \bbS_{\scaleto{\begin{array}{|c|c|}
\cline{1-2}
1 & 2 \\
\cline{1-2}
3 & 4 \\
\cline{1-2}
\end{array}}{12pt}}\, T_p^*M.
\]  

In particular, the triple
\begin{equation}\label{eq:zerlegung_von_kappa_dach}
\left ( \alpha,\beta,\gamma \right ) 
\;\;=\;\; \left (r^2\, \kappa, r^3\,\kappa^1, r^4\,C^\kappa\right )
\end{equation}
corresponds with the associated algebraic curvature tensor on the cone $S^\kappa$ defined in~\eqref{eq:def_S_kappa}.

\subsection{Symmetric Killing 2-tensors whose associated symmetrized algebraic curvature tensor is parallel}
It is natural to ask under which conditions the symmetrized algebraic curvature
tensor $S^\kappa$ on the cone $\hat M$ associated with a symmetric Killing 2-tensor
$\kappa$ on $M$ is parallel with respect to the Levi-Civita connection of
$\hat M$. More general, we wish to understand the relation between parallel 
sections of $\bbS_{\scaleto{\begin{array}{|c|c|}
\cline{1-2}
1 & 2 \\
\cline{1-2}
3 & 4 \\
\cline{1-2}
\end{array}}{12pt}}\, T^*\hat M$  and symmetric Killing 2-tensors on $M$.

Let $S$ be a symmetrized algebraic curvature tensor on the cone 
seen as a triple $(\alpha,\beta,\gamma)$ according to~\eqref{eq:glndec}
and $S|_M$ denote the restriction of $S$ to $M$ via the natural inclusion $\iota \colon M \to \hat M$. 
Hence  $S|_M$ is the triple $(\alpha|_M,\beta|_M,\gamma|_M)$ where $\alpha|_M$, $\beta|_M$ and $\gamma|_M$ are sections of 
$\bbS_{\scaleto{\begin{array}{|c|c|}
\cline{1-2}
1 & 2 \\
\cline{1-2}
\end{array}}{6pt}}\,T^*M $, $\bbS_{\scaleto{\begin{array}{|c|c|}
\cline{1-2}
2 & 3 \\
\cline{1-2}
1 & \multicolumn{1}{c}{\;\;\;} \\
\cline{1-1}
\end{array}}{12pt}}\, T^*M$ and  $\bbS_{\scaleto{\begin{array}{|c|c|}
\cline{1-2}
1 & 2 \\
\cline{1-2}
3 & 4 \\
\cline{1-2}
\end{array}}{12pt}}\, T^*M$, respectively.

The following lemma shows that parallelity of $S$ along the radial direction does
not have anything to do with the Killing equation:

\bigskip
\begin{lemma}\label{le:parallel_in_the_radial_direction}
A section $S = (\alpha,\beta,\gamma)$ of  $\bbS_{\scaleto{\begin{array}{|c|c|}
\cline{1-2}
1 & 2 \\
\cline{1-2}
3 & 4 \\
\cline{1-2}
\end{array}}{12pt}}\, T^*\hat M$ 
is parallel along the radial direction $\hat\nabla_{\pd_r} S = 0$ if and only if $S = (r^2\,\alpha|_M,r^3\, \beta|_M,r^4\,\gamma|_M)$.
\end{lemma}
\begin{proof}
Recall that the radial vector field $\pd_r$ is covariantly constant in the radial direction according to~\eqref{eq:Levi-Civita_connection_of_the_cone_1} and
that $\frac{x}{r}$ is a parallel vector field along the radial curve $r\mapsto (p,r)$ for each $x\in T_pM$, see~\eqref{eq:parallele_vektorfelder_laengs_radialer_kurven}.
Comparing this with~\eqref{eq:S_explizit}, we see that a triple  $(\alpha,\beta,\gamma)$ is covariantly constant in the radial direction $\hat\nabla_{\pd_r} S = 0$ if and only if $S = (r^2\, \alpha|_M,r^3\, \beta|_M,r^4\, \gamma|_M)$.
\end{proof}

Therefore it remains to understand the condition  $\hat\nabla_x S = 0$ for horizontal vectors $x$.
For this we recall that the vector bundle $T\hat M = \tau^*TM \oplus \R$ carries two relevant
connections: there is the Levi-Civita connection $\hat \nabla$ of the metric 
cone and also the Levi Civita connection of $M$ according to Remark~\ref{re:reduction}. Because
of~\eqref{eq:Levi-Civita_connection_of_the_cone_3},~\eqref{eq:Levi-Civita_connection_of_the_cone_2}
the difference tensor $\Delta := \hat \nabla - \nabla $ satisfies
\begin{equation}\label{eq:Delta}
\begin{array}{l} \Delta(x,y) = - r\langle x,y \rangle,\\
\Delta(x,\pd_r|_{(p,r)}) = \frac{1}{r} x
\end{array}
\end{equation}
for all $p\in M$, $r>0$ and $x\in T_pM$.

\bigskip
\begin{lemma}\label{le:parallel_in_the_horizontal_direction}
Let $S$ be a symmetrized algebraic curvature tensor of $T_{(p,r)}\hat M$ described as a triple $(\alpha,\beta,\gamma)$ according to~\eqref{eq:glndec}.
For each horizontal  vector $x\in T_pM$ the difference $\hat \nabla_xS - \nabla_xS$
is given by the triple
\begin{equation}\label{eq:second_fundamental_form_of_bbP}
\left (  \begin{array}{lll}
 - \frac{1}{r}\, x\lrcorner \beta & & \\
 - \frac{1}{r}\, x\lrcorner \gamma  & \; + 2\, r\, \alpha \owedge x^\sharp &\\
&\mbox{\ } + r\, \beta\owedge x^\sharp &
\end{array} \right ).
 \end{equation} 
\end{lemma}
\begin{proof}
We have $\hat \nabla_xS - \nabla_xS = x\lrcorner \Delta\cdot S$, where $\Delta$ is defined by~\eqref{eq:Delta} and $\cdot$ is the usual action of endomorphisms on tensors by algebraic derivation, see~\eqref{eq:derivation}. Hence,
\begin{align*}
x\lrcorner \Delta S(\partial_r,\partial_r,y_1,y_2) &\;\;=\;\; \frac{- 2}{r}\underbrace{\langle \partial_r|_p,\partial_r|_p \rangle}_{=1} S( \partial_r,x,y_1,y_2).
\end{align*}
This yields the first component in~\eqref{eq:second_fundamental_form_of_bbP}. For the second, we use that the cyclic sum of $y\lrcorner S$ vanishes for fixed $y$. 
Hence $ 2\, S(x,z,y,z) = 2\, S(y,z,x,z) = - S(y,x,z,z) = - S(z,z,x,y)$. We get
\begin{align*}
& x\lrcorner \Delta S(\partial_r,y_1,y_2,y_3) &\;\;=\;\;& \left \lbrace\begin{array}{cc}- \frac{1}{r} S(x,y_1,y_2,y_3)  & + r\, \langle x,y_1 \rangle S(\partial_r|_p,\partial_r|_p,y_2,y_3)\\
+ r\, \langle x,y_2 \rangle S(\partial_r|_p,y_1,\partial_r|_p,y_3) & + r\, \langle x,y_3\rangle S(\partial_r|_p,y_1,y_2,\partial_r|_p)  \end{array} \right \rbrace\\
& &\;\;=\;\;&  \left \lbrace\begin{array}{cc}-\frac{1}{r} S(x,y_1,y_2,y_3)  & + r\, \langle x,y_1 \rangle S(\partial_r|_p,\partial_r|_p,y_2,y_3)\\
- \frac{1}{2} r\, \langle x,y_2 \rangle S(\partial_r|_p,\partial_r|_p,y_1,y_3) & - \frac{1}{2}r\, \langle x,y_3 \rangle S(\partial_r|_p,\partial_r|_p,y_1,y_2)  \end{array} \right \rbrace\\
& &\;\;=\;\;& - \frac{1}{r} S(x,y_1,y_2,y_3) + r\, \rmS_{\scaleto{\begin{array}{|c|c|}
\cline{1-2}
2 & 3 \\
\cline{1-2}
1 \\
\cline{1-1} 
\end{array}}{12pt}}\, \langle x, y_1\rangle \alpha(y_2,y_3)\\
& & \;\;\stackrel{\eqref{eq:product_of_a_one_form_and_a_symmetric_2-tensor}}{=}\;\; &  - \frac{1}{r} \gamma(x,y_1,y_2,y_3) + 2\,r\,  \alpha\owedge x^\sharp (y_1,y_2,y_3)
\end{align*}
for all $x,y_1,\ldots,y_3\in T_pM$.  For the last component: on the one side, we have
\begin{align*}
x\lrcorner \Delta S(y_1,y_2,y_3,y_4)\;\;=\;\;&  \left \lbrace \begin{array}{c}\ \;r\, \langle x,y_1\rangle S(\partial_r|_p,y_2,y_3,y_4)
+ r\, \langle x,y_2\rangle S(y_1,\partial_r|_p,y_3,y_4)\\
+ r\, \langle x,y_3\rangle S(y_1,y_2,\partial_r|_p,y_4)
+ r\, \langle x,y_4\rangle S(y_1,y_2,y_3,\partial_r|_p)\end{array}\right \rbrace\\
\end{align*}
for all $x,y_1,\ldots,y_4\in T_pM$. On the other side,
\begin{align*}
3\,  \beta \owedge x^\sharp (x,y_1,y_2,y_3,y_4)\;\;\stackrel{\eqref{eq:noch_ein_produkt}}{=}\;\; &
\rmS_{\scaleto{\begin{array}{|c|c|}
\cline{1-2}
1 & 2 \\
\cline{1-2}
3 & 4 \\
\cline{1-2} 
\end{array}}{12pt}}\, \langle x, y_1 \rangle \beta (y_2,y_3,y_4) \\
\;\;=\;\; & \cyclic_{12}\cyclic_{34}\left ( \begin{array}{c}\langle x,y_1\rangle S(\partial_r|_p,y_2,y_3,y_4)  \\
- \langle x,y_3\rangle S(\partial_r|_p,y_2,y_1,y_4)\\
- \langle x,y_1\rangle S(\partial_r|_p,y_4,y_3,y_2)\\
+ \langle x,y_3\rangle S(\partial_r|_p,y_4,y_1,y_2)\end{array}\right ).
\end{align*}
The first and the last line of the above term yield
\begin{equation*}
\cyclic_{12}\cyclic_{34}\left ( \begin{array}{c}\langle x,y_1\rangle S(\partial_r|_p,y_2,y_3,y_4)  \\
+ \langle x,y_3\rangle S(\partial_r|_p,y_4,y_1,y_2)\end{array}\right ) \;\;=\;\;
\cyclic_{12}\cyclic_{34}\left ( \begin{array}{c}\langle x,y_1\rangle S(\partial_r|_p,y_2,y_3,y_4)  \\
+ \langle x,y_3\rangle S(y_1,y_2,\partial_r|_p,y_4)\end{array}\right )\\
\end{equation*}
which gives $\frac{2}{r}  x\lrcorner \Delta S(y_1,y_2,y_3,y_4)$. The two lines in the middle are by the Bianchi identity
\begin{align*}
- \cyclic_{12}\cyclic_{34} \left ( \begin{array}{c}
\langle x,y_3\rangle S(\partial_r|_p,y_2,y_1,y_4)\\
+ \langle x,y_1\rangle S(\partial_r|_p,y_4,y_3,y_2)
\end{array}\right ) &\;\;=\;\;& \cyclic_{12}\cyclic_{34} \left ( \begin{array}{c}
\langle x,y_3 \rangle\big ( S(y_1,y_2,\partial_r|_p,y_4) + S(\partial_r|_p,y_1,y_2,y_4)\big )\\
+ \langle x,y_1\rangle \big (S(y_3,y_4,\partial_r|_p,y_2) + S(\partial_r|_p,y_3,y_4,y_2)\big )
\end{array}\right ).
\end{align*}
Using the Bianchi identity in the form mentioned further above, this yields
\begin{align*}
\cyclic_{12}\cyclic_{34} \left ( \begin{array}{c}
\langle x,y_3 \rangle\big ( S(y_2,y_1,\partial_r|_p,y_4) - \frac{1}{2}S(y_1,y_2,\partial_r|_p,y_4)\big )\\
\langle x,y_1\rangle \big (S(y_4,y_3,\partial_r|_p,y_2) - \frac{1}{2} S(y_3,y_4,\partial_r|_p,y_2)\big )
\end{array}\right ),
\end{align*}
which adds  $\frac{1}{r}x\lrcorner \Delta S(y_1,y_2,y_3,y_4)$ to the Young symmetrizer. We conclude that 
\begin{align*}
\rmS_{\scaleto{\begin{array}{|c|c|}
\cline{1-2}
1 & 2 \\
\cline{1-2}
3 & 4 \\
\cline{1-2} 
\end{array}}{12pt}}\,\langle x, y_1 \rangle \partial_r\lrcorner S(y_2,y_3,y_4)\;\;=\;\;\frac{3}{r}x\lrcorner \Delta S(y_1,y_2,y_3,y_4).
\end{align*}
This yields the last component of~\eqref{eq:second_fundamental_form_of_bbP}.
\end{proof}

The following proposition characterizes Killing tensors whose associated
algebraic curvature tensor is parallel on the cone. This can be seen as
the analogue of~\cite[Lemma~3.2.1]{Se} for symmetric Killing 2-tensors.

\bigskip
\begin{proposition}\label{p:prolong}
Let $S$ be a symmetrized algebraic curvature tensor on the cone seen as a triple $(\alpha,\beta,\gamma)$ 
according to~\eqref{eq:glndec}. The following assertions (a) - (d) are equivalent:
\begin{enumerate}
\item  $\hat\nabla S = 0$.
\item The pair $(\kappa,C)\;:=\;(\alpha|_M,\gamma|_M)$ satisfies~\eqref{eq:alg_curv_tensor_1}-\eqref{eq:alg_curv_tensor_2}, $\nabla\kappa = \beta|_M$ and $S =S^\kappa$.
\item There exists a symmetric Killing 2-tensor $\kappa$ such that $S$ is the associated
  symmetrized algebraic curvature tensor $S^\kappa$ and the following equations hold:
\begin{align}
\label{eq:prolong_1_speziell}
\nabla_x \kappa\;\, &\;\;=\;\;x\lrcorner \kappa^1,\\
\label{eq:prolong_2_speziell}
\nabla_x\kappa^1 &\;\;=\;\;x\lrcorner \kappa^2 + x\lrcorner (g\owedge \kappa) - 2\, \kappa\owedge x^\sharp ,\\
\label{eq:prolong_3_speziell}
\nabla_x \kappa^2 &\;\;=\;\;-  g\owedge x\lrcorner \kappa^1  - \kappa^1\owedge  x^\sharp .
\end{align}
\item There exists a symmetric Killing 2-tensor $\kappa$ such that $S = S^\kappa$ and  $S^\kappa|_M$ is $\hat\nabla$-parallel.
\end{enumerate}
\end{proposition}
\begin{proof}
For $(a)\Rightarrow (b)$: if $\hat\nabla_x S|_{(p,1)} = 0$, then~\eqref{eq:second_fundamental_form_of_bbP} implies
\begin{equation*}
\left (  \begin{array}{l}
 \nabla_x \alpha \\
\nabla_x\beta\\
\nabla_x\gamma
\end{array} \right )\;\;=\;\;\left (  \begin{array}{lll}
x\lrcorner \beta & & \\
x\lrcorner \gamma  & - 2\,  \alpha\owedge x^\sharp  &\\
&\quad - \beta  \owedge x^\sharp &
\end{array} \right )
 \end{equation*}
for all $p\in M$ and $x\in T_pM$. From this the first two assertions of~(b)
follow, i.e. $S = S^\kappa$ along $r = 1$. Further, $\hat\nabla S = 0$ yields
that $S$ has the same radial dependence as $S^\kappa$ according to Lemma~\ref{le:parallel_in_the_radial_direction}.
Therefore we conclude that $S = S^\kappa$.

For $(b)\Rightarrow (c)$: The first equation~\eqref{eq:prolong_1_speziell} of~(c) is a rephrasing of $\nabla\kappa = \beta|_M$. 
Clearly, this implies that $\kappa$ is Killing. Further, 
\[
C(x_1,x_2,x_3,x_4)  \;\stackrel{\text{Example}~\ref{ex:Young_symmetrizer}~\text{(c)}}{=}\; \rmP_{\scaleto{\begin{array}{|c|c|}
\cline{1-2}
1 & 2 \\
\cline{1-2}
3 & 4\\
\cline{1-2} 
\end{array}}{12pt}} C(x_1,x_2,x_3,x_4).
\] 
Moreover, a straightforward consideration shows that
\begin{align*}
\rmP_{\scaleto{\begin{array}{|c|c|}
\cline{1-2}
1 & 2 \\
\cline{1-2}
3 & 4\\
\cline{1-2} 
\end{array}}{12pt}} \rmP_{\scaleto{\begin{array}{|c|c|}
\cline{1-2}
3 & 4 \\
\cline{1-2}
2 &  \multicolumn{1}{c}{\;\;\;} \\
\cline{1-1}  
\end{array}}{12pt}} \langle x_1, x_2\rangle \kappa(x_3,x_4)\;\;=\;\;\rmP_{\scaleto{\begin{array}{|c|c|}
\cline{1-2}
1 & 2 \\
\cline{1-2}
3 & 4\\
\cline{1-2} 
\end{array}}{12pt}} \langle x_1, x_2\rangle \kappa(x_3,x_4).
\end{align*}
Therefore applying the Young projector $\rmP_{\scaleto{\begin{array}{|c|c|}
\cline{1-2}
1 & 2 \\
\cline{1-2}
3 & 4\\
\cline{1-2} 
\end{array}}{12pt}}$ to~\eqref{eq:alg_curv_tensor_1} and recalling that the
numbers $h_{(2,1)}$ and $h_{(2,2)}$ are 3 and 12, respectively, the
definitions given in~\eqref{eq:product_of_a_one_form_and_a_symmetric_2-tensor},~\eqref{eq:def_Kulkarni_Nomizu_product}
and~\eqref{eq:kappa_prolong_2} imply that $C = C^\kappa$. Substituting~\eqref{eq:prolong_1_speziell} this
gives~\eqref{eq:prolong_2_speziell}. In the same way~\eqref{eq:alg_curv_tensor_2}
implies \eqref{eq:prolong_3_speziell}.

For $(c)\Rightarrow (d)$: Applying~\eqref{eq:second_fundamental_form_of_bbP}
to  $S := S^\kappa$ we obtain that
\begin{equation}\label{eq:second_fundamental_form_of_bbP_im_spezialfall}
\hat\nabla_xS^\kappa|_{(p,1)}\;\;=\;\;\left ( \begin{array}{llll}
\nabla_x\kappa & - x\lrcorner \kappa^1& &\\
\nabla_x \kappa^1  & - x\lrcorner \kappa^2  & - x\lrcorner (g\owedge \kappa) & + 2\, \kappa\owedge x^\sharp\\
\nabla_{x}\kappa^2 & & + \ \ \ g\owedge \nabla_x\kappa & +\ \ \kappa^1\owedge x^\sharp 
\end{array} \right ).
 \end{equation}
Hence~\eqref{eq:prolong_1_speziell}-\eqref{eq:prolong_3_speziell} imply that $\hat\nabla_xS^\kappa|_{(p,1)} = 0$ for all $p\in M$ and $x\in T_pM$.

For $(d)\Rightarrow (a)$: we have to show that $\hat\nabla S^\kappa =0 $. Covariant constancy of $S^\kappa$ 
in the radial direction follows automatically from the structure of~\eqref{eq:zerlegung_von_kappa_dach} according to Lemma~\ref{le:parallel_in_the_radial_direction}.
Further, we have by assumption $\hat\nabla_x S^\kappa|_{(p,1)} = 0$ for all $p\in M$ and $x\in T_pM$. It remains to show that $\hat\nabla_x S^\kappa|_{(p,r)} = 0$ for all $r>0$. 
According to Lemma~\ref{le:parallel_in_the_horizontal_direction}, the tensorial difference $\Delta := \hat\nabla - \nabla$ of the two connections $\hat\nabla$ and $\nabla$  satisfies 
\begin{equation*}
x\lrcorner \Delta S^\kappa|_{(p,r)} \;\;=\;\;\left (  \begin{array}{lll}
 - r^2\, x\lrcorner \nabla\kappa& & \\
 - r^3\, x\lrcorner C^\kappa  & + \ 2\, r^3\, \kappa  \owedge x^\sharp&\\
&\quad \ \  r^4\, \kappa^1\owedge x^\sharp   &
\end{array} \right )
\end{equation*} 
Therefore, the scaling of the components of $x\lrcorner \Delta S^\kappa$ in the radial parameter $r$ is the same as for $S^\kappa$. Hence 
$\nabla_xS^\kappa|_{(p,r)} = - x\lrcorner \Delta S^\kappa|_{(p,r)}$ for all $p\in M$ and $r>0$ as soon as this holds for $r = 1$. Therefore $\hat\nabla_x S^\kappa|_{(p,r)} = 0$ for all $r>0$ as soon as this is true for $r=1$.
\end{proof}

\section{Prolongation of the Killing Equation}\label{se:prolongation}
We will show through direct calculations that the triple  $(\kappa,\kappa^1,\kappa^2)$ is closed 
under componentwise covariant derivative for a Killing tensor $\kappa$.  More precisely, we consider the following curvature expressions
\begin{align}\label{eq:def_F21}
F^1(\kappa;x_1,x_2,x_3,x_4) &\;\;:=\;\; \begin{array}{cc}  \frac{1}{2}\; R_{x_1,x_2}\kappa(x_3,x_4) 
-  \frac{1}{4}\; \cyclic_{34}\cyclic_{12} R_{x_3,x_1}\kappa(x_2,x_4),
\end{array}\\
\label{eq:def_F22}
F^2(\kappa;x_1,\cdots,x_5) &\;\;:=\;\;\scaleto{\rmP_{\scaleto{\begin{array}{|c|c|}
\cline{1-2}
2 & 3 \\
\cline{1-2}
4  & 5 \\
\cline{1-2} 
\end{array}}{12pt}}}{24pt} \left (
\begin{array}{c}
   \nabla_{x_4}\big (R_{x_1,x_5}\cdot \kappa(x_2,x_3) + 2\, R_{x_2,x_5} \cdot \kappa(x_1,x_3)\big )\\
 + R_{x_1,x_5}\cdot \nabla_{x_4}\kappa(x_2,x_3) + R_{x_2,x_5}  \cdot \nabla_{x_4}\kappa(x_1,x_3)
\end{array}\right )
\end{align}
(where $\cyclic$ denote the cyclic sum). Obviously these expressions depend tensorial on $\kappa$ and $(\kappa,\nabla \kappa)$,
respectively. Then we will see that the equations of the prolongation of the
Killing equation
\begin{align}\label{eq:prolong_1}
\nabla_{x_1}\kappa (x_2,x_3)& \;\;=\;\; \kappa^1(x_1,x_2,x_3),\\
\label{eq:prolong_2}
\nabla_{x_1}\kappa^1(x_2,x_3,x_4) & \;\;=\;\; \kappa^2(x_1,x_2,x_3,x_4) + F^1(\kappa;x_1,x_2,x_3,x_4),\\
\label{eq:prolong_3}
\nabla_{x_1}\kappa^2(x_2,x_3,x_4,x_5) & \;\;=\;\; F^2(\kappa;x_1,x_2,x_3,x_4,x_5)
\end{align}
hold for every Killing tensor $\kappa$. In particular,
the triple $(\kappa,\kappa^1,\kappa^2)$ is parallel with respect to a linear
connection which is immediately clear
from~\eqref{eq:prolong_1}-\eqref{eq:prolong_3}, the Killing connection.
We already know from Lemma~\ref{le:branching_rules} that this triple can be seen as an algebraic curvature tensor 
on the direct sum $TM \oplus \R$. Because $\frac{1}{3}\binom{n+2}{2}\binom{n+1}{2}$  is the dimension of the linear space of algebraic curvature tensors in dimension $n+1$, we recover the well known result that this number 
is an upper bound for the dimension of the space of symmetric Killing 2-tensors on any Riemannian manifold (cf.~\cite[Theorem~4.3]{Th}).

The prolongation of the Killing equation for abitrary symmetric tensors 
using the adjoint Young symmetrizers and projectors is given
in~\cite{HTY}. However in comparison with~\eqref{eq:prolong_1}-\eqref{eq:prolong_3} that approach has the disadvantage of the appearance of an additional Young
symmetrizer (see~\cite[(7)-(9)]{HTY}) complicating the formulas unnecessarily.

The first equation of the prolongation~\eqref{eq:prolong_1} is a consequence
of Example~\ref{ex:Young_symmetrizer}~(b). The other two are derived as follows.

\subsection{The second equation of the prolongation}
We consider the curvature expression $F^{1}(\kappa)$ defined in~\eqref{eq:def_F21} and show that the second equation of the prolongation~\eqref{eq:prolong_2} of the Killing equation  holds.
First note that $x_1\mapsto F^1(\kappa;x_1)$ defines a 1-form on $M$ with values in the vector bundle $\bbS_{\scaleto{\begin{array}{|c|c|}
\cline{1-2}
3 & 4 \\
\cline{1-2}
2 & \multicolumn{1}{c}{} \\
 \cline{1-1}
\end{array}}{12pt}} T^*M$
and that $F^1(\kappa)$ is a tensor in $\kappa$. Further, recall that the curvature endomorphism $R_{x,y}$ acts on $\kappa$ via algebraic
derivation $R_{x,y}\cdot \kappa$. Therefore, we could allow $R$
to be an arbitrary algebraic curvature tensor to obtain a tensorial bilinear expression $F^1(\kappa,R)$.
Then we have $F^1(\kappa,R) = \tilde F^1(R \cdot \kappa)$ for some tensorial expression in the curvature tensor applied to $\kappa$.

In oder to establish~\eqref{eq:def_F21}, we need the following intermediate result:

\bigskip
\begin{lemma}
  Let $\kappa$ be Killing. Then
 \begin{eqnarray}\label{eq:Killing_alternativ_*}\cyclic_{12}\nabla^2_{x_1,x_2}\kappa(x_3,x_4) &=& \cyclic_{34} \nabla^2_{x_3,x_4}\kappa(x_1,x_2)
+ 2\, \cyclic_{34}\cyclic_{12} R_{x_3,x_1}\kappa(x_2,x_4)
\end{eqnarray}.
\end{lemma}
\begin{proof} On the one hand,
\begin{eqnarray*}
\nabla^2_{x_1,x_2}\kappa(x_3,x_4) &=& - \nabla^2_{x_1,x_3}\kappa(x_2,x_4) - \nabla^2_{x_1,x_4}\kappa(x_2,x_3)\\
                     &=&   \left \lbrace\begin{array}{c} - \big (\nabla^2_{x_3,x_1}\kappa(x_2,x_4) + \nabla^2_{x_4,x_1}\kappa(x_2,x_3) \big )\\
                                                   + R_{x_3,x_1}\kappa(x_2,x_4) +  R_{x_4,x_1}\kappa(x_2,x_3) \end{array} \right \rbrace\\
                     &=&  \left \lbrace\begin{array}{cccc} 
                                  \nabla^2_{x_3,x_4}\kappa(x_1,x_2) &+ \nabla^2_{x_3,x_2}\kappa(x_1,x_4) &+ \nabla^2_{x_4,x_3}\kappa(x_1,x_2) &+ \nabla^2_{x_4,x_2}\kappa(x_1,x_3)\\
                                 + R_{x_3,x_1}\kappa(x_2,x_4) & &+ R_{x_4,x_1}\kappa(x_2,x_3) & \end{array} \right \rbrace.
\end{eqnarray*}
On the other hand,
\begin{eqnarray*}
\nabla^2_{x_1,x_2}\kappa(x_3,x_4) -  R_{x_1,x_2}\kappa(x_3,x_4) &=& \nabla^2_{x_2,x_1}\kappa(x_3,x_4)\\
&\;\;=\;\;& \left .\begin{array}{c}  - \nabla^2_{x_2,x_3}\kappa(x_1,x_4) - \nabla^2_{x_2,x_4}\kappa(x_1,x_3)
                                 \end{array} \right .\\
&=&  \left \lbrace \begin{array}{c} - \nabla^2_{x_3,x_2}\kappa(x_1,x_4) - \nabla^2_{x_4,x_2}\kappa(x_1,x_3)\\
                             + R_{x_3,x_2}\kappa(x_1,x_4) + R_{x_4,x_2}\kappa(x_1,x_3)\end{array} \right \rbrace
\end{eqnarray*}
The result follows by adding the two expressions.
\end{proof}


\bigskip
\begin{proposition} \label{p:prolong_4}
If $\kappa$ is Killing, then~\eqref{eq:prolong_2} holds.
\end{proposition}

\begin{proof}
\begin{eqnarray*}
\rmS_{\scaleto{\begin{array}{|c|c|}
\cline{1-2}
1 & 2 \\
\cline{1-2}
3 & 4 \\
\cline{1-2}
\end{array}}{12pt}}\nabla^2_{x_1,x_2}\kappa(x_3,x_4) &\;\;=\;\;& \cyclic_{12}\; \rmS_{\scaleto{\begin{array}{|c|c|}
\cline{1-2}
4 & 3 \\
\cline{1-2}
2 & \multicolumn{1}{c}{\;\;\;} \\
\cline{1-1}
\end{array}}{12pt}} \nabla^2_{x_1,x_2}\kappa(x_3,x_4) - \nabla^2_{x_3,x_2}\kappa(x_1,x_4)\\
& \stackrel{\eqref{eq:def_Killing}}{=} & \cyclic_{12} \rmS_{\scaleto{\begin{array}{|c|c|}
\cline{1-2}
4 & 3 \\
\cline{1-2}
2 & \multicolumn{1}{c}{\;\;\;} \\
\cline{1-1}
\end{array}}{12pt}}\big (\nabla^2_{x_1,x_2}\kappa(x_3,x_4) + \nabla^2_{x_3,x_4}\kappa(x_1,x_2) +  \nabla^2_{x_3,x_1}\kappa(x_2,x_4)\big )\\
 & \stackrel{Example~\ref{ex:Young_symmetrizer}~(b)}{=} &\cyclic_{12}\big ( 3\, \nabla^2_{x_1,x_2}\kappa(x_3,x_4) + \rmS_{\scaleto{\begin{array}{|c|c|}
\cline{1-2}
4 & 3 \\
\cline{1-2}
2 & \multicolumn{1}{c}{\;\;\;} \\
\cline{1-1}
\end{array}}{12pt}}\nabla^2_{x_3,x_4}\kappa(x_1,x_2)\big ). 
\end{eqnarray*}
Further, \begin{eqnarray*}
\cyclic_{12}\; \rmS_{\scaleto{\begin{array}{|c|c|}
\cline{1-2}
4 & 3 \\
\cline{1-2}
2 & \multicolumn{1}{c}{\;\;\;} \\
\cline{1-1}
\end{array}}{12pt}} \nabla^2_{x_3,x_4}\kappa(x_1,x_2)
&\;\;=\;\;& \cyclic_{12}\cyclic_{34} \big (\nabla^2_{x_3,x_4}\kappa(x_1,x_2) - \nabla^2_{x_3,x_2}\kappa(x_1,x_4) \big )\\
 &\stackrel{\eqref{eq:def_Killing}}{=} & 3\; \cyclic_{34} \nabla^2_{x_3,x_4}\kappa(x_1,x_2)\\
& \stackrel{\eqref{eq:Killing_alternativ_*}}{=} & 3\; \cyclic_{12} \; \nabla^2_{x_1,x_2}\kappa(x_3,x_4) - 3\big (\cyclic_{12}\cyclic_{34} R_{x_3,x_1}\kappa(x_2,x_4)\big ).
\end{eqnarray*} 
Therefore
\begin{eqnarray*}
\rmP_{\scaleto{\begin{array}{|c|c|}
\cline{1-2}
1 & 2 \\
\cline{1-2}
3 & 4 \\
\cline{1-2}
\end{array}}{12pt}} \nabla^2_{x_1,x_2}\kappa (x_3,x_4) - \frac{1}{2} \cyclic_{12}\nabla^2_{x_1,x_2}\kappa(x_3,x_4) & =&
                                                       \frac{1}{4}\; \cyclic_{12}\cyclic_{34} R_{x_3,x_1}\kappa(x_2,x_4).
\end{eqnarray*}
Using the Ricci-identity,~\eqref{eq:prolong_2} follows.
\end{proof}

From the previous we obtain a Weitzenböck formula. For this recall the natural action of the curvature tensor on a covariant symmetric 2-tensor $\kappa$,
\begin{equation}\label{eq:def_of_R*R}
q(R) \kappa(x_3,x_4)\;:=\;- \sum_{i=1}^nR_{x_3,e_i}\cdot \kappa(x_4,e_i) + R_{x_4,e_i}\cdot \kappa(x_3,e_i).
\end{equation}
This is the zeroth-order term appearing in the definition of the Lichnerowicz
Laplacian on symmetric 2-tensors 
\begin{equation}\label{eq:Lichnerowicz_Laplacian}
\Delta \kappa\;:=\;\nabla^*\nabla \kappa + q(R)\kappa
\end{equation}
(cf.~\cite[1.143]{Be}). Further, if $\kappa$ is Killing, then the differential
$\d\, \trace \, \kappa$ of
its trace and the divergence $\delta\, \kappa$ are related by
\begin{equation}\label{eq:divergenz_versus_differential_der_spur}
\d\, \trace \, \kappa\;\;=\;\;2\,\delta\,\kappa.
\end{equation}

\bigskip
\begin{corollary}
Let $\kappa$ be a symmetric Killing 2-tensor. Then \begin{eqnarray}\label{eq:Wb}
 \nabla^*\nabla \kappa &\;\;=\;\;& q(R)\kappa - \nabla^2\trace \, \kappa .
\end{eqnarray}
\end{corollary}
\begin{proof}
On the one hand, taking the trace on~\eqref{eq:prolong_2} with respect to $x_3,\, x_4$ and  $x_1,\, x_2$ gives
\begin{eqnarray*}
\trace_{3,4} \big (\rmP_{\scaleto{\begin{array}{|c|c|}
\cline{1-2}
1 & 2 \\
\cline{1-2}
3  & 4 \\
\cline{1-2} 
\end{array}}{12pt}} \nabla^2\kappa (x_1,x_2,x_3,x_4)\big ) &=& - \nabla^*\nabla \kappa(x_1,x_2) - \frac{1}{4}\;\sum_{i=1}^n 
\left \lbrace\begin{array}{c} 
2\, \underbrace{R_{e_i,e_i}\kappa(x_3,x_4)}_{=0}\\
+ R_{x_1,e_i}\kappa(e_2,x_i) + R_{x_1,e_i}\kappa(e_i,x_2)\\ +  R_{x_2,e_i}\kappa(e_i,x_1) + R_{x_2,e_i}\kappa(e_i,x_1)  \end{array} \right \rbrace\\
&=&  - \nabla^*\nabla \kappa(x_1,x_2) + \frac{1}{2} q(R)\kappa(x_1,x_2)\;,
\end{eqnarray*}
and
\begin{eqnarray*}
\trace_{1,2} \big (\rmP_{\scaleto{\begin{array}{|c|c|}
\cline{1-2}
1 & 2 \\
\cline{1-2}
3  & 4 \\
\cline{1-2} 
\end{array}}{12pt}}\nabla^2\kappa (x_1,x_2,x_3,x_4)\big ) &=& \nabla^2_{x_3,x_4 }\trace \, \kappa
+ \frac{1}{4}\;\left \lbrace\begin{array}{c} 
2\, \underbrace{R_{x_3,x_4}\kappa(e_i,e_i)}_{=0}\\
+ R_{x_3,e_i}\kappa(x_4,e_i) + R_{x_4,e_i}\kappa(x_3,e_i)\\  + R_{x_3,e_i}\kappa(x_4,e_i) +
           R_{x_4,e_i}\kappa(x_3,e_i)\end{array} \right \rbrace\\
&=& \nabla^2_{x_3,x_4 }\trace \, \kappa - \frac{1}{2} q(R)\kappa(x_3,x_4).
\end{eqnarray*}
On the other hand, $\trace_{3,4} \big ( \rmP_{\scaleto{\begin{array}{|c|c|}
\cline{1-2}
1 & 2 \\
\cline{1-2}
3  & 4 \\
\cline{1-2} 
\end{array}}{12pt}}\nabla^2\kappa \big )(u,v) = \trace_{1,2} \big ( \rmP_{\scaleto{\begin{array}{|c|c|}
\cline{1-2}
1 & 2 \\
\cline{1-2}
3  & 4 \\
\cline{1-2} 
\end{array}}{12pt}}\nabla^2\kappa \big )(u,v)$ by the symmetries of a curvature tensor.
\end{proof}

\subsection{The third equation of the prolongation}
We will now consider the term $F^{2}(\kappa)$ defined in~\eqref{eq:def_F22} and show that~\eqref{eq:prolong_3} holds.

By definition, $x_1\mapsto F^2(\kappa;x_1)$ is a 1-form on $M$ with values in
the vector bundle $\bbS_{\scaleto{\begin{array}{|c|c|c|}
\cline{1-2}
2 & 4 \\
\cline{1-2}
3 & 5 \\
\cline{1-2}
\end{array}}{12pt}}T^*M$. The expression $F^2(\kappa)$ depends
linearly on $\kappa$. If we allow $R$ to be an arbitrary algebraic curvature
tensor, then we obtain a bilinear expression $F^2(\kappa,R)$. Via this
interpretation we have $F^2(\kappa,R) = \tilde
F^2(R\cdot\kappa,R\cdot\nabla\kappa)$ for some expression $\tilde F$ which depends linearly on the pair
$(R\cdot\kappa,R\cdot\nabla\kappa)$.
It is also clear from~\eqref{eq:def_F22} that $F^2(\kappa)$ is a (linear) tensor in 
$(\kappa,\nabla\kappa)$ and that $F^2(\kappa,R)$ is a bilinear tensor 
in the pairs $(\kappa,\nabla\kappa)$ and $(R,\nabla R)$. More precisely:

\bigskip
\begin{lemma}\label{le:ausmultiplizieren}
Let $M$ be Riemannian manifold $M$ with Levi Civita connection $\nabla$. Let
$\kappa$ and $R$ be a symmetric Killing 2-tensor and an algebraic curvature tensor on $M$. 
We have
\begin{equation}\label{eq:F22_ausmultipliziert}
F^2(\kappa,R;x_1,\ldots,x_5) \;\;=\;\;  \scaleto{\rmP_{\scaleto{\begin{array}{|c|c|}
\cline{1-2}
2 & 3 \\
\cline{1-2}
4  & 5 \\
\cline{1-2} 
\end{array}}{12pt}}}{24pt}\left (
\begin{array}{c}
   \nabla_{x_4} R_{x_1,x_5}\cdot\kappa(x_2,x_3) + 2\, \nabla_{x_4}R_{x_2,x_5}\cdot \kappa(x_1,x_3)\big )\\
 + 2\, R_{x_1,x_5}\cdot\nabla_{x_4}\kappa(x_2,x_3) + 3\, R_{x_2,x_5}\cdot\nabla_{x_4}\kappa(x_1,x_3)\\
-  \nabla_{R_{x_1,x_5} x_4}\cdot\kappa(x_2,x_3) - 2 \nabla_{R_{x_2,x_5}x_4}\cdot\kappa(x_1,x_3)
\end{array}\right ).
\end{equation}
\end{lemma}
\begin{proof}
The product rule for the covariant 
derivative of an endomorphism-valued tensor field $A$ acting on $\alpha$ yields
\begin{eqnarray}
\nabla_x \big (A\cdot \alpha(y_2,y_3) \big ) &\;\;=\;\;& \begin{array}{c}\nabla_x A \cdot \alpha(y_2,y_3)
+ A\cdot  \nabla_x \alpha(y_2,y_3) -  \nabla_{A\,x} \alpha(y_2,y_3)\end{array},
\end{eqnarray}
since the endomorphism should in effect not act on the covariant derivative slot. Applying this to~\eqref{eq:def_F21} we obtain~\eqref{eq:F22_ausmultipliziert}.
\end{proof}

We will prove the third equation of the prolongation of the Killing equation:

\bigskip
\begin{proposition}\label{p:prolong_3}
If $\kappa$ is Killing, then~\eqref{eq:prolong_3} holds with $F^2(\kappa)$ defined via~\eqref{eq:def_F22}.
\end{proposition}
\begin{proof}
Using the Killing equation~\eqref{eq:def_Killing} together with the fact that
\begin{equation}\label{eq:klaro}
\rmS_{\scaleto{\begin{array}{|c|c|}
\cline{1-2}
2 & 3 \\
\cline{1-2}
4 & 5 \\
\cline{1-2} 
\end{array}}{12pt}} \big ( \begin{array}{c}
\nabla^{3)}_{x_2,x_4,x_5}\kappa(x_1,x_3) + \nabla^{3)}_{x_4,x_2,x_5}\kappa(x_1,x_3)
 \end{array}\big )\;\;=\;\;0,
\end{equation}
we conclude hat
\begin{eqnarray*}
\scaleto{\rmS_{\scaleto{\begin{array}{|c|c|}
\cline{1-2}
2 & 3 \\
\cline{1-2}
4 & 5 \\
\cline{1-2} 
\end{array}}{12pt}}}{32pt} \left ( \begin{array}{c}
+ \frac{1}{2}R_{x_2,x_4} \nabla_{x_5}\kappa(x_1,x_3)\\  + \nabla_{x_4}\big (R_{x_2,x_5} \kappa (x_1,x_3)\big )\\
 + \frac{1}{2} \begin{array}{c} \nabla_{x_4}\big (R_{x_3,x_5} \kappa(x_1,x_2)\big ) \end{array}
\end{array}\right ) &=& \scaleto{\rmS_{\scaleto{\begin{array}{|c|c|}
\cline{1-2}
2 & 3 \\
\cline{1-2}
4 & 5 \\
\cline{1-2} 
\end{array}}{12pt}}}{24pt} \left ( \begin{array}{l}
\big (\nabla^{3)}_{x_2,x_4,x_5} + \nabla^{3)}_{x_4,x_2,x_5} -
                                \nabla^{3)}_{x_4,x_5,x_2}\big )\kappa(x_1,x_3)\\
 - \nabla^{3)}_{x_4,x_5,x_3}\kappa(x_1,x_2)\\
 \end{array}\right )\\
&\stackrel{\eqref{eq:klaro}}{=}& \rmS_{\scaleto{\begin{array}{|c|c|}
\cline{1-2}
2 & 3 \\
\cline{1-2}
4 & 5 \\
\cline{1-2} 
\end{array}}{12pt}}\left ( \begin{array}{l}
- \nabla^{3)}_{x_4,x_5,x_2}\kappa(x_1,x_3) - \nabla^{3)}_{x_4,x_5,x_3}\kappa(x_1,x_2) \\
 \end{array}\right )\\
&\stackrel{\eqref{eq:def_Killing}}{=}& \rmS_{\scaleto{\begin{array}{|c|c|}
\cline{1-2}
2 & 3 \\
\cline{1-2}
4 & 5 \\
\cline{1-2} 
\end{array}}{12pt}} \nabla^{3)}_{x_4,x_5,x_1}\kappa(x_2,x_3)
\end{eqnarray*}
Thus the Ricci identity
\begin{equation*}
\nabla^{3)}_{x_1,x_4,x_5}\kappa(x_2,x_3)  - \nabla^{3)}_{x_4,x_5,x_1}\kappa(x_2,x_3)\;\;=\;\;R_{x_1,x_4}\nabla_{x_5}\kappa(x_2,x_3) + \nabla_{x_4} \big (R_{x_1,x_5} \kappa (x_2,x_3)\big )
\end{equation*}
shows that
\begin{eqnarray*} 
 \nabla_{x_1}\rmP_{\scaleto{\begin{array}{|c|c|}
\cline{1-2}
2 & 3 \\
\cline{1-2}
4 & 5 \\
\cline{1-2} 
\end{array}}{12pt}}\nabla^2_{x_4,x_5}\kappa(x_2,x_3) &\;\;=\;\;&\scaleto{\rmP_{\scaleto{\begin{array}{|c|c|}
\cline{1-2}
2 & 3 \\
\cline{1-2}
4 & 5 \\
\cline{1-2} 
\end{array}}{12pt}}}{32pt}\left (
\begin{array}{c}
 R_{x_1,x_4}\nabla_{x_5}\kappa(x_2,x_3)
+ \nabla_{x_4} \big ( R_{x_1,x_5} \kappa(x_2,x_3)\big )\\
 + \nabla_{x_4}\big (R_{x_2,x_5} \kappa(x_1,x_3) \big)\\ 
+ \frac{1}{2}\left ( \begin{array}{c}\nabla_{x_4}\big (R_{x_3,x_5} \kappa(x_1,x_2)\big )
+ R_{x_2,x_4} \nabla_{x_5}\kappa(x_1,x_3)\end{array} \right )
\end{array}\right )
\end{eqnarray*}
The skew-symmetry $R_{u,v} = - R_{v,u}$ implies that
\begin{eqnarray*}
\frac{1}{2}\rmS_{\scaleto{\begin{array}{|c|c|}
\cline{1-2}
2 & 3 \\
\cline{1-2}
4 & 5 \\
\cline{1-2} 
\end{array}}{12pt}}\nabla_{x_4} \big (R_{x_3,x_5} \kappa(x_1,x_2)\big ) &=& \rmS_{\scaleto{\begin{array}{|c|c|}
\cline{1-2}
2 & 3 \\
\cline{1-2}
4 & 5 \\
\cline{1-2} 
\end{array}}{12pt}}\nabla_{x_4}\big (R_{x_2,x_5}\kappa(x_1,x_3)\big ),
\end{eqnarray*}

\begin{eqnarray*}
\frac{1}{2}\rmS_{\scaleto{\begin{array}{|c|c|}
\cline{1-2}
2 & 3 \\
\cline{1-2}
4 & 5 \\
\cline{1-2} 
\end{array}}{12pt}} R_{x_2,x_4} \nabla_{x_5} \kappa(x_1,x_3) &=& \rmS_{\scaleto{\begin{array}{|c|c|}
\cline{1-2}
2 & 3 \\
\cline{1-2}
4 & 5 \\
\cline{1-2} 
\end{array}}{12pt}} R_{x_2,x_5} \nabla_{x_4} \kappa(x_1,x_3) 
\end{eqnarray*}
and the symmetry of $\kappa$ gives
\begin{eqnarray*}
\rmS_{\scaleto{\begin{array}{|c|c|}
\cline{1-2}
2 & 3 \\
\cline{1-2}
4 & 5 \\
\cline{1-2} 
\end{array}}{12pt}} R_{x_1,x_4}\nabla_{x_5} \kappa(x_2,x_3) &=& \rmS_{\scaleto{\begin{array}{|c|c|}
\cline{1-2}
2 & 3 \\
\cline{1-2}
4 & 5 \\
\cline{1-2} 
\end{array}}{12pt}}R_{x_1,x_5}\nabla_{x_4}\kappa(x_2,x_3).
\end{eqnarray*}
This establishes that~\eqref{eq:prolong_3} holds with $F^2(\kappa)$ defined
by~\eqref{eq:def_F22}.
\end{proof}

\subsection{Prolongation  of symmetric Killing 2-tensors in spaces of constant
curvature}
Let $M$ be a (pseudo) Riemannian manifold of constant sectional
curvature. Since the following arguments are purely local, we can assume that
$M$ is a (generalised) sphere as considered in
Example~\ref{ex:constant_curvature}. Then the cone is a flat (pseudo)
Euclidean space, hence Proposition~\ref{p:prolong} shows that here the
dimension of the space of Killing tensors is the maximal possible. In particular,~\eqref{eq:prolong_1_speziell}-\eqref{eq:prolong_3_speziell} describe
the prolongation of symmetric Killing 2-tensors on $M$. This enables us
to calculate $F^1(\kappa,R_1)$ and $F^2(\kappa,R_1)$ with $R_1(x,y)\;:=\;-
x\wedge y$ on every pseudo Riemannian manifold $M$ as follows.

\bigskip
\begin{lemma}\label{le:curvature_terms_for_S1}
We have
\begin{align}\label{eq:F21_speziell}
 x\lrcorner F^1(R_1,\kappa)\;\;=\;\;& x\lrcorner (g\owedge \kappa) \ - 2\, \kappa\owedge x^\sharp ,\\
\label{eq:F22_speziell}
 x\lrcorner F^2(R_1,\kappa)\;\;=\;\;& \ \ \ -  g\owedge x\lrcorner \kappa^1  - \kappa^1 \owedge x^\sharp .
 \end{align}
for every symmetric 2-tensor $\kappa$ on $M$.
\end{lemma}
\begin{proof}
We have already remarked that $F^1$ is a linear tensor in $R$ and that $F^2$ depends
tensorial linear on the 1-jet of $R$. Therefore, and because $\nabla R_1 = 0$ on every
pseudo Riemannian manifold, it suffices to show that~\eqref{eq:F21_speziell} and~\eqref{eq:F22_speziell} hold on a generalised sphere $M$. Further, comparing the equations of the prolongation of the
Killing equation~\eqref{eq:prolong_2}-\eqref{eq:prolong_3} with~\eqref{eq:prolong_1_speziell}-\eqref{eq:prolong_3_speziell},
we see that~\eqref{eq:F21_speziell},\eqref{eq:F22_speziell} hold for every Killing tensor. Moreover, since both $F^1$ and $F^2$
depend tensorial linear on the 1-jet of $\kappa$ and because moreover every 1-jet of a symmetric 2-tensor on $M$ can be extended 
to a Killing tensor according to Proposition~\ref{p:prolong}, these identities
automatically hold for all symmetric 2-tensors $\kappa$.
\end{proof}

For a proof of the following in the Riemannian case see~\cite[Lemme~1.2]{Ga},
the indefinite case does not need an extra argument.

\bigskip
\begin{lemma}\label{le:krümmungstensor_des_kegels}
  The curvature tensor of $\hat M$ is a horizontal tensor. When seen as a
  $(1,3)$-tensor it does not depend on the radial component $r$ and is given by
\begin{equation}
\label{eq:hat_R_1}
\hat R_{x,y}z\;\;=\;\;R_{x,y}z + x\wedge y (z)
\end{equation}
for all $p\in M$ and $x,y,z\in T_pM$.
\end{lemma}
We obtain the following extension of Proposition~\ref{p:prolong}:

\bigskip
\begin{proposition}\label{p:main}
The following conditions are equivalent:
\begin{enumerate}
\item
The algebraic curvature tensor $S^\kappa$ associated with $\kappa$ is parallel on the cone $\hat M$;
\item $\kappa$ is Killing and the curvature conditions~\eqref{eq:curv_cond_1} and~\eqref{eq:curv_cond_2} hold.
\end{enumerate}
\end{proposition}
\begin{proof}
For $(a) \Rightarrow (b)$: since $\hat \nabla S^\kappa = 0$, the curvature
endomorphism $\hat R_{x,y}$ annihilates $S$ for all $p\in M$ and $x,y\in T_pM$. Because  $\hat R$ is horizontal, we conclude that $\hat R$ annihilates
the components of~\eqref{eq:zerlegung_von_kappa_dach}. In particular, $\hat R_{x,y}\cdot \kappa = 0$ and $\hat R_{x,y} \cdot \nabla\kappa = 0$ for all $p\in M$ and $x,y\in T_pM$. 
Using~\eqref{eq:hat_R_1}, we see that both~\eqref{eq:curv_cond_1} and~\eqref{eq:curv_cond_2} hold.

For $(b)\Rightarrow (a)$: by assumption,~\eqref{eq:curv_cond_1}
and~\eqref{eq:curv_cond_2} hold. The expression $F^1(\kappa,R)$ is explicitly described in~\eqref{eq:def_F21}. From this one sees that it depends in fact only on the action $R_{x,y}\cdot \kappa$ of
the given algebraic curvature tensor $R$ on $\kappa$. Therefore~\eqref{eq:curv_cond_1} implies that
$F^1(R,\kappa) = F^1(R_1,\kappa)$. Similarly, the explicit
description~\eqref{eq:def_F22} of $F^2(\kappa,R)$ implies that the latter expression depends only on the actions 
$R_{x,y}\cdot\kappa$ and $R_{x,y}\cdot \nabla \kappa$ of the algebraic curvature tensor $R$ on $\kappa$ and $\nabla\kappa$. 
We obtain from~\eqref{eq:curv_cond_1} and~\eqref{eq:curv_cond_2} that $F^2(\kappa,R) = F^2(\kappa,R_1)$. Then we conclude from Lemma~\ref{le:curvature_terms_for_S1} that $F^1(\kappa,R)$ and $F^2(\kappa,R)$ are given by~\eqref{eq:F21_speziell} and~\eqref{eq:F22_speziell}. Substituting this into~\eqref{eq:prolong_1}-\eqref{eq:prolong_3} we see that~\eqref{eq:prolong_1_speziell}-\eqref{eq:prolong_3_speziell} hold. Hence the result follows from Proposition~\ref{p:prolong}.
\end{proof}

\section{Proof of Theorem~\ref{th:main_1}}
\label{se:implications}
Because of Propositions~\ref{p:prolong} and~\ref{p:main}, it remains to prove the following proposition:

\bigskip
\begin{proposition}\label{p:get_rid}
Suppose that $\dim(M) \geq 2$.
\begin{enumerate}
\item 
If the conditions~\eqref{eq:alg_curv_tensor_1}-\eqref{eq:alg_curv_tensor_2} hold for a pair $(\kappa,C)$, then $\kappa$ is
Killing. Further, $\kappa = \kappa^C$ where
\begin{equation}\label{eq:def_kappaC}
\kappa^C \;:=\; - \frac{1}{2}\widetilde \Ric^C + \frac{\widetilde \scal^C}{2(n - 1)}g
\end{equation}
and
\begin{equation}\label{eq:def_modified_Ricci_trace_und_tilde_scalC}
\widetilde \Ric^C \;\;:=\;\; \Ric^C + \frac{1}{4}\nabla^2\scal^C\ \ \ \text{and}\ \ \
\widetilde \scal^C \;\;:=\;\; \scal^C - \frac{1}{4} \nabla^*\nabla\scal^C.
\end{equation}
\item Suppose that $M$ is compact with negative definite metric tensor. If ~\eqref{eq:alg_curv_tensor_1}-\eqref{eq:alg_curv_tensor_2} hold for
a pair $(\kappa,C)$, then there exists a constant $c$ such that $\kappa = c\, g$. 
\end{enumerate}
\end{proposition}

In the previous proposition and hence also in Theorem~\ref{th:main_1} the condition  $\dim(M) > 1$ can not be neglected:

\bigskip
\begin{remark}\label{re:main_3}
On the real line $\R$ the pair $(\kappa,C)$ where $\kappa$ is a quadratic form with affine linear
coefficients and $C\;:=\;0$ matches~\eqref{eq:alg_curv_tensor_1}-\eqref{eq:alg_curv_tensor_2} of Theorem~\ref{th:main_1}. 
But $\kappa$ is neither Killing nor $\kappa = \kappa^C$ holds.
\end{remark}
For the proof of Proposition~\ref{p:get_rid}, we make the following general assumption for the rest of this section:

\begin{center}
 \flushleft
\begin{small}
\bf Let a pair $(\kappa,C)$ be given such that~\eqref{eq:alg_curv_tensor_1} and~\eqref{eq:alg_curv_tensor_2} hold for some constant $c$ different from zero. 
\end{small}
\end{center}
The proof of the previous proposition requires several lemmas.

\bigskip
\begin{lemma}\label{le:Bianchis_cyclic_sum_is_parallel}
The cyclic sum $\cyclic_{123} \nabla_{y_1} \kappa (y_2,y_3)$ is parallel, i.e. 
\begin{equation}\label{eq:Bianchis_cyclic_sum_is_parallel}
 \cyclic_{123} \nabla^2_{x,y_1} \kappa (y_2,y_3)\;\;=\;\;0
\end{equation}
for all $x,y_1,y_2,y_3\in T_pM$ and $p\in M$.
\end{lemma}
\begin{proof}
We have
\[
 \cyclic_{123} C(x,y_1,y_2,y_3)\;\;=\;\;0
\]
for every symmetrized algebraic curvature tensor. Further, $\cyclic_{123}\beta(x_1,x_2,x_3) = 0$ for every $\beta\in \bbS_{\scaleto{\begin{array}{|c|c|}
\cline{1-2}
2 & 3 \\
\cline{1-2}
1 \\
\cline{1-1} 
\end{array}}{12pt}}T_pM^*$. Therefore,
\begin{align*}
\cyclic_{123}\, \rmP_{\scaleto{\begin{array}{|c|c|}
\cline{1-2}
2 & 3 \\
\cline{1-2}
1 \\
\cline{1-1} 
\end{array}}{12pt}}\,\langle x, y_1\rangle \kappa (y_2,y_3)\;\;=\;\;0.
\end{align*} Thus~\eqref{eq:Bianchis_cyclic_sum_is_parallel} follows from~\eqref{eq:alg_curv_tensor_1}.
\end{proof}

\bigskip
\begin{lemma}
The algebraic Ricci tensor of $C$ and its scalar trace are given by
\begin{align}\label{eq:Ricci_spur_von_C_1}
\Ric^C(x,y) &\;\;=\;\;\nabla^2_{x,y}\trace \, \kappa + 2(\langle x,y\rangle \trace \, \kappa - \kappa(x,y))\\
\label{eq:Ricci_spur_von_C_2}
&\;\;=\;\;- \nabla^*\nabla \kappa(x,y) + 2(n-1) \kappa(x,y),\\
\label{eq:skalare_spur_von_C}
\scal^C &\;\;=\;\;-\nabla^*\nabla\trace \, \kappa + 2(n-1)\trace \, \kappa.
\end{align}
\end{lemma}
\begin{proof}
For~\eqref{eq:Ricci_spur_von_C_1}, take a suitable trace of~\eqref{eq:alg_curv_tensor_1}. From this~\eqref{eq:skalare_spur_von_C} follows immediately.
\end{proof}

Let $\delta\kappa$ and $\trace \, \kappa$ denote the divergence and trace of $\kappa$. 

\bigskip
\begin{lemma}
We have
\begin{align}\label{eq:kovariante_Ableitung_des_Ricci_tensors_von_C}
- 3 \nabla_{x} \Ric^C(y_1,y_2) &\;\;=\;\;2 \cyclic_{12}\left \lbrace  \begin{array}{c}\big (\langle x,y_1\rangle \d \, \trace \, \kappa(y_2) - \nabla_{y_1} \kappa(y_2,x)\\
+\langle x,y_1\rangle \delta\kappa(y_2) \end{array} \right \rbrace + 4 \nabla_x \kappa(y_1,y_2),\\
\label{eq:differential_der_skalaren_spur_von_C}
 - 3\d \scal^C &\;\;=\;\;8\big (\d\, \trace \, \kappa + \delta\kappa \big ).
\end{align}
\end{lemma}
\begin{proof}
Recalling the
definition of the product $\owedge$ from~\eqref{eq:noch_ein_produkt}, the
negative of r.h.s. of~\eqref{eq:alg_curv_tensor_2} is explicitly given as a
Young symmetrized expression
\begin{align*}
& \langle x,x_1\rangle \nabla_{x_2} \kappa(x_3,x_4) - \langle x,x_3\rangle \nabla_{x_2} \kappa(x_1,x_4) - \langle x,x_1\rangle \nabla_{x_4} \kappa(x_3,x_2) + \langle x,x_3\rangle \nabla_{x_4} \kappa(x_1,x_2)\\
& \langle x,x_2\rangle \nabla_{x_1} \kappa(x_3,x_4) - \langle x,x_3\rangle \nabla_{x_1} \kappa(x_2,x_4) - \langle x,x_2\rangle \nabla_{x_4} \kappa(x_1,x_3) + \langle x,x_3\rangle \nabla_{x_4} \kappa(x_1,x_2)\\
&  \langle x,x_1\rangle \nabla_{x_2} \kappa(x_3,x_4) - \langle x,x_4\rangle \nabla_{x_2} \kappa(x_1,x_3) - \langle x,x_1\rangle \nabla_{x_3} \kappa(x_4,x_2) + \langle x,x_4\rangle \nabla_{x_3} \kappa(x_1,x_2)\\
& \langle x,x_2\rangle \nabla_{x_1} \kappa(x_3,x_4) - \langle x,x_4\rangle \nabla_{x_1} \kappa(x_2,x_3) - \langle x,x_2\rangle \nabla_{x_3} \kappa(x_1,x_4) + \langle x,x_4\rangle \nabla_{x_3} \kappa(x_1,x_2)
\end{align*}
Taking the trace with respect to $x_3$ and $x_4$ yields
\begin{align*}
& \langle x,x_1\rangle \nabla_{x_2} \kappa(e_i,e_i) - \langle x,e_i\rangle \nabla_{x_2} \kappa(x_1,e_i) - \langle x,x_1\rangle \nabla_{e_i} \kappa(e_i,x_2) + \langle x,e_i\rangle \nabla_{e_i} \kappa(x_1,x_2)\\
& \langle x,x_2\rangle \nabla_{x_1} \kappa(e_i,e_i) - \langle x,e_i\rangle \nabla_{x_1} \kappa(x_2,e_i) - \langle x,x_2\rangle \nabla_{e_i} \kappa(x_1,e_i) + \langle x,e_i\rangle \nabla_{e_i} \kappa(x_1,x_2)\\
& \langle x,x_1\rangle \nabla_{x_2} \kappa(e_i,e_i) - \langle x,e_i\rangle \nabla_{x_2} \kappa(x_1,e_i) - \langle x,x_1\rangle \nabla_{e_i} \kappa(e_i,x_2) + \langle x,e_i\rangle \nabla_{e_i} \kappa(x_1,x_2)\\
& \langle x,x_2\rangle \nabla_{x_1} \kappa(e_i,e_i) - \langle x,e_i\rangle \nabla_{x_1} \kappa(x_2,e_i) - \langle x,x_2\rangle \nabla_{e_i} \kappa(x_1,e_i) + \langle x,e_i\rangle \nabla_{e_i} \kappa(x_1,x_2)\\
\end{align*}
Thus~\eqref{eq:kovariante_Ableitung_des_Ricci_tensors_von_C} follows from~\eqref{eq:alg_curv_tensor_2}. Taking the trace of this yields~\eqref{eq:differential_der_skalaren_spur_von_C}.
\end{proof}

\bigskip
\begin{lemma}
We have
\begin{align}\label{eq:erste_vorstufe_zur_Bianchi_identity}
& 2\,\cyclic_{123} \nabla_{x_1} \kappa (x_2,x_3)\;\;=\;\;\left \lbrace \begin{array}{c} 3\nabla^3_{x_1,x_2,x_3}\trace \, \kappa +  6\, \d \, \trace \, \kappa(x_1) \langle x_2,x_3\rangle\\ 
+ 2\,\cyclic_{2,3}\langle x_1,x_2\rangle \delta\kappa(x_3)\\
 + 2\,\cyclic_{2,3}\langle x_1,x_2 \rangle\, \d \, \trace \, \kappa(x_3) \end{array} \right \rbrace.
\end{align} 
\end{lemma}
\begin{proof}
Solving~\eqref{eq:kovariante_Ableitung_des_Ricci_tensors_von_C} for the cyclic sum of $\nabla \kappa$ yields:
\begin{align*}
& 2\,\cyclic_{123} \nabla_{x_1} \kappa (x_2,x_3)\;\;=\;\;\left \lbrace \begin{array}{c} 3\, \nabla_{x_1} \Ric^C(x_2,x_3) +  6\, \nabla_{x_1} \kappa(x_2,x_3)\\ 
+ 2\,\cyclic_{23}\langle x_1,x_2\rangle \delta  \kappa(x_3)\\
 + 2\,\cyclic_{23}\langle x_1,x_2 \rangle\, \d \, \trace \, \kappa(x_3) \end{array} \right \rbrace.
\end{align*}
Substituting now the formula~\eqref{eq:Ricci_spur_von_C_1} for the algebraic Ricci tensor of $C$ 
gives~\eqref{eq:erste_vorstufe_zur_Bianchi_identity}.
\end{proof}

\bigskip
\begin{lemma}\label{le:M_is_irreducible}
If $\kappa$ is not a constant multiple of the metric tensor, then $M$ is irreducible.
\end{lemma}
\begin{proof}
Assume by contradiction that $M$ is not irreducible. Then $M$ is the product $M_1 \times M_2$ of two Riemannian manifolds of dimensions at least one.
Let $(p,q)\in M_1 \times M_2$ and $(x,y)\in T_pM_1\oplus T_qM_2$. 
Since $x$ and $y$ are tangent to different factors, the covariant derivatives commute $\nabla^2_{y,x}\kappa  =  \nabla^2_{x,y}\kappa$. Using~\eqref{eq:Bianchis_cyclic_sum_is_parallel}, we thus have
\begin{align*}
0 &\stackrel{\eqref{eq:Bianchis_cyclic_sum_is_parallel}}{=} \nabla^2_{y,x}\kappa(x,x)\;\;=\;\; \nabla^2_{x,y}\kappa(x,x)\stackrel{\eqref{eq:Bianchis_cyclic_sum_is_parallel}}{=} - 2\,\nabla^2_{x,x}\kappa(x,y)
\end{align*}
for all $x\in T_pM_1$ and $y\in T_qM_2$. Hence, the exchange rule $C(x,x,x,y) = 0$ implies from~\eqref{eq:alg_curv_tensor_1} that
\begin{align*}
0 \;\;=\;\; - \langle x, x\rangle \kappa (x,y) + \underbrace{\langle x, y\rangle \kappa (x,x)}_{=0}
\end{align*}
We conclude that $ \kappa (x,y) = 0$, i.e. $\kappa = \kappa_1 + \kappa_2$ where $\kappa_1$ and $\kappa_2$ are the restrictions of 
$\kappa$ to $T_pM_1\times T_pM_1$ and $T_qM_2\times T_qM_2$, respectively. Further, the symmetry $C(x,y,x,y) = C(y,x,x,y)$ 
and the commuting of the covariant derivatives
$\nabla^2_{x,y}\kappa(x,y) = \nabla^2_{y,x}\kappa(x,y)$ yields
\begin{equation*}
\langle x, x\rangle \kappa (y,y)\;\;=\;\;\langle y, y\rangle \kappa (x,x)
\end {equation*}
Hence $\kappa(x,x) = \kappa(y,y)$ for all unit vectors $x\in T_pM_1$ and $y\in T_qM_2$. 
Thus there exists a function $f \colon M_1\times M_2 \to \R_+$  such that
$\kappa_1(x,x) = f(p,q) = \kappa_2(y,y)$ for all unit vectors $x$ and $y$, i.e. $\kappa =  f\langle\ ,\ \rangle$.  Inserting this into~\eqref{eq:Bianchis_cyclic_sum_is_parallel} implies that $(n + 2)\d f$ is parallel, hence
$\nabla \d f = 0$. Substituting this into~\eqref{eq:erste_vorstufe_zur_Bianchi_identity} yields
\begin{align*}
& \cyclic_{123}\, \d_{x_1}f \langle x_2,x_3 \rangle\;\;=\;\;3\,n\, \langle x_2,x_3\rangle \d_{x_1}f + (n-1) \cyclic_{23}\langle x_1,x_2 \rangle\, \d_{x_3}f.
\end{align*}
Evaluating this equation with $x_1 = x_2 = x_3$ implies that $(n - 1) \d f = 0$. Therefore if $n > 1$ then $\kappa$ is a constant multiple of the metric tensor. We conclude that $M$ is irreducible.
\end{proof}

\bigskip
\begin{corollary}\label{co:contracted_Bianchi_identity}
If $\dim(M) \geq  2$, then the contraction~\eqref{eq:divergenz_versus_differential_der_spur} of the Killing equation holds.
\end{corollary}
\begin{proof} 
We can assume that $\kappa$ is not a constant multiple of the metric tensor,
since otherwise~\eqref{eq:divergenz_versus_differential_der_spur} is obvious. Then, on
the one hand, $M$ is irreducible by the previous lemma and hence the holonomy
group acts fixed point free on tangent vectors,  because $\dim(M)\geq 2$. On
the other hand, since the contraction
of~\eqref{eq:Bianchis_cyclic_sum_is_parallel} implies that the 1-form
$\lambda\;:=\;\d \, \trace \, \kappa - 2 \delta \kappa$ is parallel, the latter is fixed under the holonomy group at each point. We conclude that $\lambda = 0$.
\end{proof}

\bigskip
\begin{corollary}
If $\dim(M) \geq  2$, then we have
\begin{align}\label{eq:zweite_vorstufe_zur_Bianchi_identity}
2 \cyclic_{123} \nabla_{y_2} \kappa (y_1,y_3) &\;\;=\;\;\left \lbrace \begin{array}{c} 3\nabla^3_{y_1,y_2,y_3}\trace \, \kappa +  6\, \d \, \trace \, \kappa(y_1) \langle y_2,y_3\rangle  + 3\,\cyclic_{23}\langle y_1,y_2 \rangle\, \d \, \trace \, \kappa(y_3) \end{array} \right \rbrace,\\
\label{eq:vergleich_der_spur_von_kappa_mit_der_skalaren_kruemmung_von_C}
\d\scal^C &\;\;=\;\;-4\, \d\, \trace \, \kappa.
\end{align}
\end{corollary}
\begin{proof} This follows by substituting~\eqref{eq:divergenz_versus_differential_der_spur} into~\eqref{eq:erste_vorstufe_zur_Bianchi_identity} and~\eqref{eq:differential_der_skalaren_spur_von_C}, respectively.
\end{proof}

Let $R_1$ denote the algebraic curvature tensor defined by
\begin{equation}
R_1(x,y,z)\;:=\;- \langle x , z \rangle y + \langle y , z \rangle x
\end{equation}
for all $x,y,z\in T_pM$. If $g$ is positive definite, then this is the curvature tensor of a round
sphere of radius one.

Recall that the $1$-nullity of a pseudo Riemannian manifold is the subspace of $T_pM$ defined by
\begin{equation}\label{eq:curvature_constancy}
\scrC_1(T_pM)\;:=\;\{x\in T_pM| R(x,y,z)= R_1(x,y,z)\}.
\end{equation}
Hence a vector $x$ belongs to $\scrC_1(T_pM)$ if and only if the sectional curvature of any two-plane in $T_pM$ containing $x$ is equal to one (see~\cite{Gray1},~\cite{MoSe1}).
For example, a Sasakian vector field $\xi$ belongs at each point to the
$1$-nullity distribution. 

In general, $\scrC_1(TM)\;:=\;\bigcup_{p\in M}\scrC_1(T_pM)$ is not a subbundle of $TM$. However, for a symmetric space $\nabla R = 0$ implies that $\scrC_1(TM)$ is invariant under parallel translation and hence a parallel vector subbundle of $TM$.

\bigskip
\begin{lemma}\label{le:curvature_constancy}
If $\dim(M) \geq  2$, then the gradient of $\trace \, \kappa$ belongs to the 1-nullity $\scrC_1(TM)$ everywhere.
\end{lemma}
\begin{proof}
Let $f := \trace \, \kappa$. Since l.h.s. of~\eqref{eq:zweite_vorstufe_zur_Bianchi_identity} is completely
symmetric in $u_1,u_2,u_3$, the same is true for the right hand side of that
equation. Hence
\begin{align*}
& 3\, \big (\nabla^3_{y_1,y_2,y_3}f + 2\, \langle y_2,y_3\rangle \d f(y_1)
+ \cyclic_{23}\langle y_1,y_2\rangle \d f(y_3) \big )\\
&\;\;=\;\; \cyclic_{123} \big (\nabla^3_{y_1,y_2,y_3}f + 2\, \langle y_2,y_3\rangle \d f(y_1)
+ \cyclic_{23}\langle y_1,y_2\rangle  \d f(y_3) \big )\\
&\;\;=\;\;\cyclic_{123}\big (\nabla^3_{y_1,y_2,y_3}f +  4\,\langle y_1,y_2\rangle  \d f(y_3)\big )\\
&\;\;=\;\;\nabla^3_{y_1,y_2,y_3}f + \nabla^3_{y_3,y_1,y_2}f + \nabla^3_{y_2,y_3,y_1}f  + 4\,\cyclic_{123}\,\langle y_1,y_2\rangle \d f(y_3).
\end{align*}
Using the symmetry of the Hessian $\nabla\d f$, we have
\begin{align*}
 \nabla^3_{y_3,y_1,y_2}f &\;\;=\;\;\nabla^2_{y_1,y_3}\d f(y_2) + R_{y_3,y_1} \d f(y_2)\;\;=\;\;\nabla^3_{y_1,y_2,y_3}f + R_{y_3,y_1} \d f(y_2),\\
 \nabla^3_{y_2,y_3,y_1}f &\;\;=\;\; \nabla^3_{y_2,y_1,y_3}f\;\;=\;\;\nabla^2_{y_1,y_2,y_3}f + R_{y_2,y_1} \d f(y_3).
\end{align*}
Further,
\begin{equation*}
 6\, \langle y_2,y_3\rangle \d f(y_1)
+ 3\,\cyclic_{23}\langle y_1,y_2\rangle \d f(y_3) - 4\,\cyclic_{123}\,\langle y_1,y_2\rangle \d f(y_3)\;\;=\;\;2\, \langle y_2,y_3\rangle
\d f(y_1) - \cyclic_{23}\langle y_1,y_2\rangle \d f(y_3)
\end{equation*}
We conclude that
\begin{align*}
& 2\, \langle y_2,y_3\rangle
\d f(y_1) - \cyclic_{23}\langle y_1,y_2\rangle  \d f(y_3)
\;\;=\;\;R_{y_3,y_1}  \d f(y_2) + R_{y_2,y_1} \d f(y_3) .
\end{align*}
Setting $y_1 := x$ and $y_2 := y_3 := y$ we obtain that
\begin{equation*}
2\, \d f(R(x,y,y))\;\;=\;\;2\, \langle y,y\rangle
\d f(x) - 2\, \langle x,y\rangle \d f(y).
\end{equation*}
This shows that the gradient of  $\trace \, \kappa$ belongs to $\scrC_1(TM)$.
\end{proof}

\begin{proof} [Proof of Proposition~\ref{p:get_rid}~(a)]
Suppose that~\eqref{eq:alg_curv_tensor_1}-\eqref{eq:alg_curv_tensor_2} hold for a pair $(\kappa,C)$ and that $\dim(M)\geq 2$. 
We show that the cyclic sum $\cyclic_{xyz} \nabla_x\kappa(y,z)$ vanishes at
each point $p\in M$. We distinguish two cases:
If there exists an open neighbourhood around $p$ where $\trace \, \kappa$ is constant, then $\kappa$ is Killing on this neighbourhood according
to~\eqref{eq:zweite_vorstufe_zur_Bianchi_identity}. 

Suppose there exists a point $p\in M$ such that $\d \, \trace \, \kappa|_p\neq 0$. 
In particular, $\kappa$ is not a constant multiple of the metric tensor on any open neighbourhood of $p$.
Therefore each neighbourhood of $p$ in $M$ is irreducible according to Lemma~\ref{le:M_is_irreducible}. Hence we can assume that $M$ is simply connected and irreducible. Further, recall that the cyclic sum $\cyclic_{xyz}\nabla_x\kappa(y,z)$ is a parallel symmetric 3-tensor according to Lemma~\ref{le:Bianchis_cyclic_sum_is_parallel}. Therefore if the holonomy group of $M$ at $p$ acts transitively on the unit sphere of $T_pM$,
 then the corresponding homogeneous polynomial function $f(x) := \nabla_x\kappa(x,x)$ with $x\in T_pM$ is constant on vectors of equal length. 
Since the  degree of $f$ is odd, we see that $f(x) = f(-x) = - f(x)$ for all $x\in T_pM$ . Therefore $\cyclic_{xyz} \nabla_x\kappa(y,z)|_p = 0$. 
If, by contradiction, we assume that the holonomy group acts non-transitively, then $M$ is locally symmetric,  $\nabla R = 0$, 
according to Bergers classification of holonomy groups~\cite[5.21]{Ba}. Hence
the $1$-nullity~\eqref{eq:curvature_constancy} is a parallel  subbundle of
$TM$. Since $\scrC_1(T_pM)\neq \{0\}$ because of
Lemma~\ref{le:curvature_constancy}, it is non-trivial and thus $\scrC_1(TM) = TM$ by the irreducibility of $M$. This implies that $M$ has constant
sectional curvature, in particular the holonomy group is $\SO(n)$ which acts
transitively on the unit sphere of $T_pM$. But this is contrary to our assumptions.

So far we have shown that  $\cyclic_{xyz}\nabla_x\kappa(y,z)|_p = 0$ at
all points $p$ of $M$ where $\d \, \trace \, \kappa|_p \neq 0$ or $\trace \, \kappa$ is constant on an open neighbourhood of $p$. Obviously the set of these points is dense in $M$. Since moreover the set points where the cyclic sum $\cyclic_{xyz}\nabla_x\kappa(y,z)$ vanishes is closed, we conclude that $\kappa$ is Killing.

We finish the proof of~(a) of Proposition~\ref{p:get_rid} by showing that automatically $\kappa = \kappa^C$: we use~\eqref{eq:Ricci_spur_von_C_1} and~\eqref{eq:vergleich_der_spur_von_kappa_mit_der_skalaren_kruemmung_von_C} to obtain that
\begin{align*}
\widetilde \Ric^C(x,y) \stackrel{\eqref{eq:def_modified_Ricci_trace_und_tilde_scalC}}{=} \Ric^C(x,y) + \frac{1}{4} \nabla^2_{x,y}\scal^C
                   \stackrel{\eqref{eq:Ricci_spur_von_C_1},\eqref{eq:vergleich_der_spur_von_kappa_mit_der_skalaren_kruemmung_von_C}}{=} 2 \big (\langle x,y\rangle \trace \, \kappa - \kappa(x,y)\big ).
\end{align*}
Hence
\begin{equation*}
\widetilde\scal^C\;\;=\;\;2(n - 1)\trace \, \kappa.
\end{equation*}
Inserting the previous into~\eqref{eq:def_kappaC} shows that $\kappa = \kappa^C$. This finishes the proof of Part~(a) of  Proposition~\ref{p:get_rid}.
\end{proof}

\begin{proof} [Proof of Proposition~\ref{p:get_rid}~(b)]
Let a pair $(\kappa,C)$ with~\eqref{eq:alg_curv_tensor_1}-\eqref{eq:alg_curv_tensor_2} on a compact manifold $M$ 
with negative definite metric tensor $g$ be given. We claim that $\kappa$ is a constant multiple of $g$: 
If $M$ is a 1-dimensional sphere, then the space of algebraic curvature tensors is trivial. 
Hence~\eqref{eq:prolong_2_speziell} and~\eqref{eq:prolong_3_speziell} together
show that the third covariant derivative of the pullback $\pi^*\kappa$ of
$\kappa$ under the canonical covering $\pi:\R\to M$ vanishes. 
This implies polynomial growth of $\pi^*\kappa$. But $\pi^*\kappa$ is periodic, hence the previous forces that $\pi^*\kappa$ is constant, 
i.e. $\kappa$ is trivial. Hence we can assume that $\dim(M)\geq 2$. Using that $\kappa$ is Killing, we conclude
from~\eqref{eq:zweite_vorstufe_zur_Bianchi_identity} that $f\;:=\;\trace \, \kappa$ satisfies Gallots equation~\eqref{eq:Gallots_Gleichung}. Taking in this equation the trace with respect to $y_1$ and $y_2$, we obtain for the rough Laplacian
\begin{equation}\label{eq:Laplacian_of_f}
\nabla^*\nabla \ df \;\;=\;\; (n + 3)  \d f.
\end{equation}
 Since on a compact manifold with negative definite metric the rough Laplacian $\nabla^*\nabla$ is a strictly negative operator, we
 conclude that $\d\, \trace \, \kappa = 0$, i.e. $\trace \, \kappa$ is a constant $c$.
Using~\eqref{eq:Ricci_spur_von_C_1} and~\eqref{eq:Ricci_spur_von_C_2} we thus see that 
\begin{equation*}
\nabla^*\nabla \big (n\, \kappa - c\, g\big )\;\;=\;\;2\, \, (n\, \kappa
- c\,g\big ).
\end{equation*} Using again the compactness of $M$, 
we conclude that $\kappa = \frac{c}{n}\, g$.
This shows that $\kappa$ is a constant multiple of the metric tensor.
\end{proof}

\section{Proof of Theorem~\ref{th:main_2}}
\label{se:sasaki}
For the proof of Theorem~\ref{th:main_2}, it remains to analyze the space of parallel algebraic curvature tensors on the cone over a Riemannian manifold. 
This is purely a question of holonomy. Gallot showed that the cone over a complete Riemannian manifold is irreducible unless the universal covering of $M$ is isometric to a round sphere. Moreover, the cone $\hat M^{n+1}$ can never be Einstein unless it is Ricci-flat, since the curvature tensor $\hat R$ is purely horizontal according to~\eqref{eq:hat_R_1}. In the following we can assume that $M$ is simply connected and complete. 

The remaining possible holonomy algebras are according to Bergers classification~\cite[5.21]{Ba}: $\so(n + 1)$, $\u(m)$ or $\su(m)$
for $n + 1 = 2 m$ and $m \ge 2$, $\sp(m)$ for $n + 1 = 4 m$ and $m \ge 2$ or $\mathfrak{g}_2$ for $n + 1= 7$
    or $\spin_7$ for $n + 1= 8$.
Holonomy algebras $\su(m)$, $\u(m)$  and  $\sp(m)$ of the cone
correspond to (Einstein)-Sasakian and 3-Sasakian manifolds, respectively.  For a Kähler 
manifold with Kähler form $\omega$, the Cartan product $\omega\odot\omega$
clearly is a non-trivial parallel algebraic curvature tensor and similar every
Hyperkähler manifold with Hyperkähler structure  $\{\omega_1,\omega_2,\omega_3\}$
admits six linearly independent non-trivial parallel algebraic curvature
tensors $\omega_i \odot \omega_j$.

By the holonomy principle every parallel section
of the vector bundle of algebraic curvature tensors
on $\hat M$ corresponds to an element of the vector space of algebraic
curvature tensors on some tangent space that is fixed by the holonomy
algebra. For the proof of Theorem~\ref{th:main_2}, we thus have to show that
\begin{itemize}
\item the actions of $\frakg_2$, $\spin_7$ and  $\so(n+1)$ on algebraic curvature tensors in dimensions $7$, $8$ and $n+1$ each only have a one-dimensional trivial component,
\item the trivial component of $\su(m)$ and $\u(m)$ acting on $\scrC(\R^{2m})$ is 2-dimensional, 
\item and the trivial component of  $\sp(m)$  on $\scrC(\R^{4m})$ is seven-dimensional.
\end{itemize}

 By irreducibility of $\hat M$, every symmetric 2-tensor is a multiple of the metric, i.e. the space of traceless symmetric 2-tensors $\Sym^2(V^*)_0$ remains irreducible over the holonomy algebra. The next two lemmas will show that the space of algebraic curvature tensors $\scrC(T)$ remains irreducible over  $\spin_7$ or $\lieG2$. 
This finishes the proof of the classification of Riemannian manifolds matching the statement of Theorem~\ref{th:main_2}.

The following result can be found in~\cite[(4.7)]{KQ}:

\bigskip
\begin{lemma} \label{lem:sasaki:g2_trivial}
  Let $V$ be an irreducible $\so(7)$ representation that contains a
  $\lieG2$-trivial subspace, then there is a $k \in \I{N}$ such that $V$ is
  the $k$-fold Cartan product of the spin representation $\Delta$. That is $V = \Delta^{\odot k}$.
  In particular, the $\so(7)$-representation space space of Weyl tensors $\scrW(V)$ yields no further trivial component when restricted to $\lieG2$.
\end{lemma}

The result for $\spin_7 \subset \so(8)$ is similar. Let $\Delta^{\pm}$ be
the spin representations of $\so(8)$. Take a non-vanishing element $s \in
\Delta^+$ and consider its isotropy algebra $\lie{h} = \set{X\suchthat X s =
0} \subset \so(8)$. Then $\lie{h} = \spin_7$ is isomorphic to the holonomy
algebra of a manifold $M$ with $\Spin_7$ holonomy \cite[Chapter IV,
§10]{LM}.
Note that changing orientation on $M$ interchanges the spin representations
$\Delta^+$ and $\Delta^-$ \cite[Theorem~3.6.1]{Jo}.

\bigskip
\begin{lemma} \label{lem:sasaki:spin7_trivial}
  Let $V$ be an irreducible $\so(8)$ representation that contains a
  $\spin_7$-trivial subspace, then there is a $k \in \I{N}$ such that $V$ is
  the $k$-fold Cartan product of the spin representation $\Delta^+$. That is
  $V = {\Delta^+}^{\odot k}$.
\end{lemma}
\begin{proof}
  The fundamental representations of $\so(8)$ are the standard representation
  $\hat\kappa \colon\ \so(8) \to \End(\I{R}^8)$, the adjoint representation
  $\ad \colon\ \so(8) \to \End \parens{\Exterior^2 \I{R}^8}$ and the spinor
  representations $\hat\kappa^{\pm} \colon \so(8) \to \End(\I{R}^8)$.
  Furthermore the triality automorphism $\sigma$  of $\so(8)$ is an outer automorphism
  of order three that interchanges the representations $\hat\kappa$ and
  $\hat\kappa^\pm$, that is $\hat\kappa \circ \sigma = \hat\kappa^-$ and $\hat\kappa^- \circ \sigma =
  \hat\kappa^+$, while $\ad \circ \sigma$ is equivalent to $\ad$~\cite[Chapter 1 §8]{LM},~\cite[§20.3]{FH}.
  The image of the standard embedding of $\so(7) \subset \so(8)$ is the
  isotropy algebra
  \begin{equation}
    \so(7)\;\;=\;\;\set{X \in \so(8) \suchthat \hat\kappa(X) e_1 = 0}
  \end{equation}
  while representatives of the other two conjugacy classes of $\spin_7$ are
  given as
  \begin{equation}
    \spin^\pm_7
   \;\;=\;\;\set{X \in \so(8) \suchthat \hat\kappa^{\pm}(X) e_1 = 0}.
  \end{equation}
  As described above, taking $\hat\kappa^+$ yields the holonomy algebra of a manifold
  with holonomy $\Spin_7$.
  Because $\hat\kappa = \hat\kappa^+ \circ \sigma$ it immediately follows that $\so(7) =
  \sigma^{-1}(\spin^+_7)$. By that, the branching rules of an $\so(8)$-representation 
$\tau \colon\ \so(8) \to \End(V)$ to the subalgebra
  $\spin^+_7$ are the same as the branching rules of $\tau \circ \sigma$ to
  $\so(7)$.  The latter rules are stated in. 
  It follows from~\cite[(25.34),(25.35)]{FH} that the highest weight
  of an $\so(8)$-irreducible representation $V$ containing a non-trivial
  $\so(7)$-trivial subspace is necessarily of the form $(k,0,0,0)$ so $V$ is the
  $k$-fold Cartan product $T^{\odot k}$ of the standard representation.
  On the other hand, the automorphism $\sigma$ permutes the fundamental weights
  and hence maps simple roots to simple roots. By \cite[Theorem~5.5
  (a)]{Kn} and the formula given in \cite[Theorem~5.5 (d)]{Kn},
  $\sigma$ maps every highest weight of a representation to a highest weight.
  Hence, the automorphism commutes with the Cartan product.
  This proves that every $\so(8)$-irreducible representation $V$ containing
  a $\spin^+_7$-trivial subspace has to be ${\Delta^+}^{\odot k}$.
\end{proof}

It is left to calculate the multiplicity of the $H$-trivial
subrepresentation in $\scrC(\hat T) $ for the remaining holonomy
algebras and compare it to the number of parallel sections obtained by taking
Cartan products of $S_\circ$ and the Kähler forms $\omega_I$, $\omega_J$ and
$\omega_K$. This can be done by the branching rules given in \cite{Ho}.
For this we have to consider the complexification
$\scrC(\hat T)  \otimes \I{C} = \scrC(\hat T\otimes \C)$ which is a $\GL(n+1, \I{C})$ representation.
If the holonomy $H$ of the cone $\hat M$ is $\U(m)$ or
$\SU(m)$ with $n+1 = 2m$, then the complexifications of $H$-representations
are $\GL(m,\I{C})$- or $\SL(m,\I{C})$-representations where
$\GL(m, \I{C}) \subset \Sp(2m, \I{C}) \subset \GL(2m, \I{C})$. If $H = \Sp(m)$
with $4m=n+1$, then complexifications of $\Sp(m)$-representations are
$\Sp(2m,\I{C})$-representations where $\Sp(2m,\I{C}) \subset \GL(2m,\I{C})
\subset \Sp(4m,\I{C}) \subset \GL(4m,\I{C})$. The Branching rules for $\GL(m, \I{C})\subset\Sp(2m,\I{C})$ and 
$\Sp(2m, \I{C}) \subset \GL(2m, \I{C})$ can be found in~\cite{Ho}. From these one easily deduces the next two lemmas:

\bigskip
\begin{lemma} \label{lem:sasaki:um}
  For $m \ge 2$, consider the embeddings $\GL(m,\I{C}) \subset \Sp(2m,\I{C})
  \subset \GL(2m,\I{C})$. There are two copies of the trivial $\GL(m,\I{C})$-representation contained in
  $\scrC(\hat T\otimes \C)$. Furthermore, if $m>2$, any other
  non-trivial one-dimensional $\GL(m,\I{C})$-representation is not
  contained in $\scrC(\hat T)$ so that this branching rule is
  the same for $\SL(m,\I{C})$.
\end{lemma}

\bigskip
\begin{lemma} \label{lem:sasaki:spm}
  For $m \ge 2$ and $\Sp(2m,\I{C}) \subset \GL(2m,\I{C}) \subset \Sp(4m,\I{C})
  \subset \GL(4m,\I{C})$ there are seven copies of the trivial $\Sp(2m,\I{C})$-representation contained in $\scrC(\hat T\otimes \C)$.
\end{lemma}

Moreover, the assumption $m>2$ in the second part of Lemma~\ref{lem:sasaki:um} is not really restrictive. Namely every 3-dimensional Einstein-Sasakian manifold is 3-Sasakian in accordance with $\SU(2) \simeq \Sp(1)$. Further, every 3-dimensional 3-Sasakian manifold is of constant sectional curvature one according to~\cite[Proposition~1.1.2~(iii)]{BG} and hence locally isometric to the 3-dimensional standard sphere. This finishes the proof of Theorem~\ref{th:main_2}.

\section{Proof of Corollary~\ref{co:Gallot}}
\label{se:Gallot}
Let us first recall Gallots original proof of Corollary~\ref{co:Gallot} (cf.~\cite[Corollary~3.3]{Ga}). For every function $f$ on $M$ we
define the function $F(p,r)\;:=\;r^2 f(p)$ and the symmetric two-tensor $q :=
\frac{1}{2}\hat\nabla^2 F$ on the cone $\hat M$. Then we can recover the original function via 
\begin{equation}\label{eq:kanonische_konstruktion_von_funktionen}
f(p)\;\;=\;\;q(\partial_r|_{(p,1)},\partial_r|_{(p,1)}).
\end{equation}
Further, we know that~\eqref{eq:Gallots_Gleichung} holds if and only if $\hat
\nabla q = 0$. As mentioned already before, he also showed that the cone over
$M$ is irreducible unless the universal covering of $M$ is a Euclidean sphere $\rmS^n$. If the cone is irreducible, then 
its holonomy group acts irreducible on the tangent spaces of $\hat M$. Therefore, $q$ is a constant multiple of the metric tensor of $\hat M$ according to Schurs Lemma, i.e. $f$ is constant.

To put Gallots result in order with the theory developed in this article, we consider the symmetric 2-tensor $\kappa$ defined by~\eqref{eq:Gallots_kappa}, 
the symmetrized algebraic curvature tensor $S := q \owedge \hat g$ on $\hat M$ and its pullback to $M$ 
\begin{equation}
C\;\;:=\;\;f\, \langle\ ,\ \rangle \owedge \langle\ ,\ \rangle + \frac{1}{2}
\nabla^2 f \owedge \langle\ ,\ \rangle.
\end{equation} 
If~\eqref{eq:Gallots_Gleichung} holds, then $q$ is $\hat\nabla$-parallel by Gallots original calculations and hence $S$ shares this property. 
Further, it is straightforward that $\kappa$ defined by~\eqref{eq:Gallots_kappa} is related to $S$ via~\eqref{eq:kanonische_konstruktion_von_killingtensoren} or that 
the pair $(\kappa,C)$ is a solution to~\eqref{eq:alg_curv_tensor_1}-\eqref{eq:alg_curv_tensor_2}. 

Either way, it follows from Theorem~\ref{th:main_2} that $M$ is a Sasakian
manifold or a sphere. To show that $M$ is a sphere, according to Corollary~\ref{co:main_2} it suffices to show that the trace of $\kappa$ is not constant unless $f$ is: we have
\begin{align*}
&\trace \, \kappa\;\;=\;\;n\, f - \frac{1}{4} \nabla^*\nabla f,\\
&\d\, \trace \, \kappa\;\;=\;\;n\, \d f  - \frac{1}{4} \d\, \nabla^*\nabla f.
\end{align*}
Evaluating~\eqref{eq:Gallots_Gleichung} we see that
\[
\d\, \trace \, \kappa\;\;=\;\;\frac{3\, n+1}{2} \d f.
\]
Hence, if $\d f\neq 0$, then we know from Corollary~\ref{co:main_2} that $M$ is a
sphere.\qed

\paragraph{}ACKNOWLEDGMENTS. We would like to thank Uwe Semmelmann for some helpful critic comments on the form of this article. We would also like to express our gratitude 
to Gregor Weingart for sharing his ideas about Killing tensors.
\ \ \

\bibliographystyle{amsplain}

\begin{thebibliography}{999}


\bibitem{Ba} H. Baum:
\newblock Eichfeldtheorie, 
\newblock Springer Berlin, Heidelberg (2009)

\bibitem{BCO} J.~Berndt, S.~Console, C.~Olmos:
\newblock Submanifolds and holonomy,
\newblock Chapman \& Hall CRC (2003)


\bibitem{Be} A.~Besse:
\newblock Einstein manifolds,
\newblock Ergeb. Math. Grenzgeb. {\bf 3. Folge Band 10}, Springer (1987).

\bibitem{BCEG} T.~Branson, A.~Cap, M.~Eastwood, R.~Gover:
\newblock Prolongations of geometric overdetermined systems,
\newblock International Journal of Mathematics, Volume 17 {\bf 6}, 641 -- 664 (2006)


\bibitem{BG} C. Boyer, K. Galicki:
\newblock 3-Sasakian Manifolds in: ``Surveys in differential geometry''. 
\newblock Cambridge MA. Internat. Press of Boston {\bf 6}, 123 -- 184 (1999)

\bibitem{F} W.~Fulton:
\newblock Young Tableaux. With Applications to Representation Theory and Geometry,
\newblock London Mathematical Society Student Texts  {\bf 35}, Cambridge University Press.

\bibitem{FH} W.~Fulton,  J.~Harris:
\newblock Representation Theory: A First Course,
\newblock Graduate Texts in Mathematics/Readings in Mathematics {\bf 129}, Springer-Verlag, New York (1991).


\bibitem{Ga} S.~Gallot:
\newblock Equations differentielles caracteristiques de la sphere,
\newblock Ann. scient. Ec. Norm. Sup. 4. {\bf 12},  235 -- 267 (1979).

\bibitem{Gray1} A.~Gray:
\newblock Spaces of Constancy of Curvature Operators,
\newblock  AMS Vol. 17, {\bf 4}, pp. 897 -- 902 (1966).

\bibitem{Ha} T. Halverson:
\newblock Characters of the centralizer algebras of mixed tensor representations of $Gl(r,\C)$ and the quantum group $\rmU_q(\gl(r, C))$, 
\newblock Pacific J. Math. 174 {\bf 2}, 359–410  (1996).


\bibitem{HM} I.~Hauser, R.~J.~Malhiot:
\newblock Structural equations for Killing tensors of order two. I,
\newblock J. Math. Phys. {\bf 16}, 150 -- 152 (1975)


\bibitem{He} K.~Heil:
\newblock Killing and conformal Killing tensors,
\newblock Dissertation an der Universität Stuttgart (2018).


\bibitem{HMS} K.~Heil, A.~Moroianu, U.~Semmelmann:
\newblock Killing and Conformal Killing tensors,
\newblock J.~Geom.~Phys. {\bf 106}, 383 -- 400 (2016).

\bibitem{HTY} T.~Houri,~ K.~Tomoda, Y.~Yasui:
\newblock On the integrability of the Killing equation,
\newblock arXiv:1704.02074v2 (2017)

\bibitem{Ho} R. E. Howe, E. C. Tan, and J. F. Willenbring: 
\newblock Stable branching rules for classical symmetric pairs,
\newblock Trans. Amer. Math. Soc. 357, {\bf 4}, 1601 -- 1626 (2005).

\bibitem{Jo} D.~Joyce:
\newblock Compact Manifolds with Special Holonomy,
\newblock Oxford University Press (2000).




\bibitem{KQ} R. C. King and A. H. A. Qubanchi: 
\newblock The branching rule for the restriction from SO(7) to G2,
\newblock J. Phys. Math. Gen. 11 {\bf 1}, 1 -- 7 (1978).

\bibitem{Kn} A.~W.~Knapp: 
\newblock Lie Groups Beyond an Introduction,
\newblock Birkhäuser (1996). 


\bibitem{LM} H.~B.~Lawson and M.-L.~Michelsohn:
\newblock Spin Geometry,
\newblock Princeton University Press (1990).

\bibitem{McMS} R.~G.~McLenaghan, R.~Milson, R.~G.~Smirnov:
\newblock Killing tensors as irreducible representations of the general linear group,
\newblock  C. R. Math. Acad. Sci. Paris~339, {\bf no. 9}, 621 -- 624 (2004)


\bibitem{MoSe1} A.~Moroianu, U.~Semmelmann:
\newblock Clifford structures on Riemannian manifolds,
\newblock  Adv. in Math. 228, {\bf 2}, pp. 940 -- 967 (2011).

\bibitem{Sch} K.~Sch\"obel:
\newblock Algebraic integrability conditions for Killing tensors on constant sectional curvature manifolds,
\newblock J.~Geom.~Phys.~62, {\bf no. 5}, 1013 -- 1037 (2012).

\bibitem{Se} U.~Semmelmann:
\newblock Conformal Killing forms on Riemannian manifolds,
\newblock  Math. Z. 245, {\bf no. 3}, 503 -- 527  (2003).


\bibitem{Th} G.~Thompson:
\newblock Killing tensors in spaces of constant curvature,
\newblock J.~Math.~Phys. 27, {\bf no. 11}, 2693 -- 2699  (1986).

\bibitem{W} T.~Wolf:
\newblock Structural equations for Killing tensors of arbitrary rank,
\newblock Comput.~Phys.~Comun. 115, {\bf no. 2-3}, 316 -- 329  (1998).

\end{thebibliography}

\vspace{2cm}
\begin{center}

 \qquad
 \parbox{60mm}{Konstantin Heil\\
  Mathematisches Institut\\
  Universit{\"a}t zu Stuttgart\\
  Paffenwaldring 57\\
  D-70569 Stuttgart, Germany\\[1mm]
  \texttt{Konstantin.Heil@mathematik.uni-stuttgart.de}}

 \qquad
 \parbox{60mm}{Tillmann Jentsch\\
  Mathematisches Institut\\
  Universit{\"a}t zu Stuttgart\\
  Paffenwaldring 57\\
  D-70569 Stuttgart, Germany\\[1mm]
  \texttt{Tillmann.Jentsch@mathematik.uni-stuttgart.de}}
\end{center}

\end{document}